\documentclass[sn-mathphys]{sn-jnl}

\jyear{2022}%

\usepackage{hyperref}
\usepackage{soul}
\usepackage{mathtools}
\usepackage{amsmath}
\usepackage{mathtools,amssymb,stackengine}

\newcommand\xLeftRightarrow[2][]{\ensurestackMath{\mathrel{%
  \stackengine{1pt}{%
    \stackengine{0pt}{\Longleftrightarrow}{\scriptstyle#2}{O}{c}{F}{F}{S}%
  }{\scriptstyle#1}{U}{c}{F}{F}{S}%
}}}

\theoremstyle{thmstyleone}%
\newtheorem{theorem}{Theorem}
\newtheorem{proposition}[theorem]{Proposition}%
\newtheorem{corollary}[theorem]{Corollary}%
\newtheorem{lemma}[theorem]{Lemma}

\theoremstyle{thmstyletwo}%
\newtheorem{remark}{Remark}%

\theoremstyle{thmstylethree}%
\newtheorem{definition}{Definition}%

\begin{document}

\title[Compactification of Perception Pairs and Spaces of GENEOs]{Compactification of Perception Pairs and Spaces of Group Equivariant non-Expansive Operators}


\author*[]{\fnm{Faraz} \sur{Ahmad}}\email{faraz.ahmad2@unibo.it}

\affil[]{\orgdiv{Dipartimento di Matematica}, \orgname{Universit\`a di Bologna}, \orgaddress{\country{Italia}}}



\abstract{We define the notions of a compact perception pair, compactification of a perception pair, and compactification of a space of group equivariant non-expansive operators. We prove that every perception pair with totally bounded space of measurements, which is also rich enough to endow the common domain with a metric structure, can be isometrically embedded in a compact perception pair. Likewise, we prove that if the images of group equivariant non-expansive operators in a given space form a cover for their common codomain, then the space of such operators can be isometrically embedded in a compact space of group equivariant non-expansive operators, such that the new reference perception pairs are compactifications of the original ones having totally bounded data sets. Meanwhile, we state some compatibility conditions for these embeddings and show that they too are satisfied by our constructions.}



\maketitle


\section{Introduction}\label{Intro}

The importance of equivariance in machine learning is widely recognized. The use of equivariant operators allows one to incorporate domain knowledge into the learning process and introduce symmetries in data space, thereby paving the way not only to speeding up machine learning and reducing large dimensionality of data but also to the introduction of new abstract representations \cite{Invar, EqCNN, HarNet, ReprLearn, RoBeToHo20}.

From the epistemological perspective, equivariant operators can be interpreted as observers that transform data into (usually simpler and more interpretable) data.
In our mathematical framework, we are interested in data observers that are represented by functional operators transforming data in a regular and stable way, while respecting the compatibility with the action of an underlying group $G$ of transformations, which describes the equivalence between data \cite{MaIntPa, PosPa}. The essence of group equivariant operators lies in their commutativity with respect to the action of $G$, and one of the most important regularity, \textit{viz.} non-expansivity, enables one to avoid instability and divergent behavior. 

Our research focuses on the study of topological properties of these \textit{group equivariant non-expansive operators} (GENEOs, for short). Such operators can be seen as components of a new kind of neural networks as well as selected observers whose expertise is leveraged to improve data analysis.
The use of GENEOs opens new possibilities in applications. For example, a shallow and interpretable neural network based on GENEOs, \textit{viz.} GENEOnet, has been recently proposed for the efficient detection of protein pockets that can host ligands \cite{GENEOProteins}.

In some sense, GENEOs constitute a bridge between geometric deep learning \cite{BrBrLeSzVa17,Br22} and topological data analysis. They make available a mathematical model for the concepts of agent and observer, seen from a geometrical perspective. Moreover, they present interesting links with persistent homology and allow one to get lower bounds for the natural pseudo-distance associated with the action of a group of homeomorphisms \cite{MaIntPa}. Furthermore, the concept of GENEO is useful in the architectural analysis of neural networks. Therefore, it is natural to study the metric and topological properties of the spaces of GENEOs. This study, coupled with our compactification results, could prove useful for the research in artificial intelligence.

Formally speaking, GENEOs are maps between so-called \emph{perception pairs} $(\Phi,G)$, where $\Phi$ is a set of bounded real-valued maps defined on a non-empty set $X$ and $G$ is a group of $\Phi$-preserving bijections of $X$. The space $\Phi$ represents the signals or measurements that the observer can interpret, while $G$ is the equivariance group associated with the action of the observer.
The space $\Phi$ is naturally endowed with a metric structure and endows $X$ and $G$ with suitable pseudo-metrics or metrics. 
This reflects the epistemological assumption that any information (and hence any quantitative structure) follows from physical measurements. It is interesting to observe that some of the pseudo-metric and topological properties of $\Phi$ are propagated to $X$ and $G$, but not all. For example, if $\Phi$ is totally bounded, then so are $X$ and $G$ \cite[Theorem 1, Theorem 4]{NQTh}, while there are simple examples of perception pairs with compact $\Phi$ but incomplete $X$ and $G$ \cite{MaIntPa}. 

Compactness results provide us with fundamental guarantees in machine learning. It is known that the space of all group equivariant non-expansive operators associated with a given group homomorphism is compact whenever the spaces of signals are compact \cite[Theorem 7]{MaIntPa}. In some sense, this result states that if the spaces of data are compact, then the space of observers is compact too, provided that suitable topologies are used. This ensures that, for any specified tolerance, there is always a finite set of GENEOs in the space that can approximate the behavior of each GENEO within any acceptable proximity. 

Therefore, it is natural to seek embeddings of important mathematical structures into compact ones. This process is called \textit{compactification} in general topology. Formally, a compact Hausdorff space $K$ is a compactification of a given space $A$ if it contains a dense subspace $D$ homeomorphic to $A$. In the case of metric spaces, we require the underlying homeomorphism $e: A \to D \subseteq K$ to be an isometry.

In view of the widely recognized importance of compactifications, we seek conditions under which a  given space $\mathcal{F} \subseteq \mathcal{F}^{\mathrm{all}}_T$ of GENEOs $(F, T) : (\Phi, G) \to (\Psi, H)$, and the respective perception pairs $(\Phi,G)$, $\mathrm{dom}(\Phi) = X$ and $(\Psi,H)$, $\mathrm{dom}(\Psi) = Y$, can be embedded isometrically into compact ones, where $\mathcal{F}^{\mathrm{all}}_T$ denotes the topological space of all GENEOs between the perception pairs $(\Phi,G)$, $(\Psi,H)$, with respect to the homomorphism $T:G\to H$. In this article, we ascertain which spaces of GENEOs allow us to construct the surrounding compact spaces of GENEOs isometrically containing the original ones. We prove that, in many practical applications, every perception pair and an important class of spaces of GENEOs can be viewed as parts of compact perception pairs and compact spaces of GENEOs.

We will be assuming that our data sets $\Phi$ and $\Psi$ are totally bounded and are rich enough to endow $X$ and $Y$, and therefore $G$ and $H$ respectively, with a metric structure. Moreover, we will also assume that the collection $\{F(\Phi) \mid F \in \mathcal{F} \subseteq \mathcal{F}^{\mathrm{all}}_T\}$ covers the data set $\Psi$.

Our approach, in brief, is as follows. The total boundedness of $\Phi$ ensures that $X$ is totally bounded (Theorem \ref{PhiBddXBdd}), and therefore, its metric completion $\hat{X}$ is compact. 
We extend the functions $\varphi \in \Phi$ to functions $\hat{\varphi} : \hat{X} \to \mathbb{R}$ on the metric completion $\hat{X}$ (Subsection \ref{ExSg}), and use the isometries $g \in G$ to define the isometries $\hat{g} : \hat{X} \to \hat{X}$ (Subsubsection \ref{CpG}). The set $\hat{\Phi}$, being isometric to the totally bounded space $\Phi$ is likewise totally bounded (Corollary \ref{CpPhiTBdd}), while the set $\hat{G}$ of all $\hat{g}$ may or may not be compact despite $\Phi$ being totally bounded \cite{MaIntPa}. We, therefore, consider their closures $\overline{\hat{\Phi}}$ and $\overline{\hat{G}}$ in the complete space $C(\hat{X}, \mathbb{R})$ and the compact space $\mathrm{Iso}(\hat{X})$ of isometries of $\hat{X}$ respectively, and, constructing successively the perception pairs $(\hat{\Phi}, \hat{G}), (\overline{\hat{\Phi}}, \hat{G})$, and $(\overline{\hat{\Phi}}, \overline{\hat{G}})$, we obtain the compatible embedding of the original perception pair $(\Phi, G)$ into the compact perception pair $(\overline{\hat{\Phi}},\overline{\hat{G}})$ (Subsubsection \ref{subSectEPPT}). 
If $\mathcal{F}$ is a space of GENEOs $(F, T) : (\Phi, G) \to (\Psi, H)$, then these perception pairs allow us to define two suitable spaces $\mathcal{F}_1 \subseteq \mathcal{F}^{\mathrm{all},1}_{\hat{T}}$ and $\mathcal{F}_2 \subseteq \mathcal{F}^{\mathrm{all},2}_{\hat{T}}$ of GENEOs $(\hat{
F}, \hat{T}) : (\hat{\Phi}, \hat{G}) \to (\hat{\Psi}, \hat{H})$ and $(\overline{\hat{
F}}, \hat{T}) : (\overline{\hat{\Phi}}, \hat{G}) \to (\overline{\hat{\Psi}}, \hat{H})$ respectively  (Subsubsection \ref{CpF} and Section \ref{CompF}). Under the covering assumption stated above, we can define a suitable space $\mathcal{F}_3 \subseteq \mathcal{F}^{\mathrm{all}}_{\overline{\hat{T}}}$ of GENEOs $(\overline{\hat{
F}}, \overline{\hat{T}}) : (\overline{\hat{\Phi}}, \overline{\hat{G}}) \to (\overline{\hat{\Psi}}, \overline{\hat{H}})$, while the closure of $\mathcal{F}_3 = \{ \overline{\hat{F}} : \overline{\hat{\Phi}} \to \overline{\hat{\Psi}} \mid F \in \mathcal{F} \}$ in the compact space $\mathcal{F}^{\mathrm{all}}_{\overline{\hat{T}}}$ serves as the requisite compactification of the space $\mathcal{F} \subseteq \mathcal{F}^{\mathrm{all}}_T$ (Section \ref{CompF}).

While the literature concerning equivariant neural networks is already extensive, the topological research about them is still quite limited. Until now, most of the attention has been devoted to what is called \emph{topological machine learning}; i.e., the joint use of topology-based methods and machine learning algorithms \cite{HeMoRi21}, in general terms. In this field, some research focuses on the study of so-called \emph{intrinsic topological features}, which concerns the employment of topological features to analyze or influence the machine learning model. 
In particular, some regularisation techniques have been considered, such as \emph{topological autoencoders} \cite{HoKwNiDi19, MoHoRiBo20} (based on the idea of building networks that can simplify the data without changing their topology) or methods to simplify the topological complexity of the decision boundary \cite{ChNiBaWa19}. More fundamental principles of regularisation using topological features have been investigated in \cite{HoGrNiKw20}. The inclusion of topological features of graph neighborhoods into a standard graph neural network (GNN) has been proposed in \cite{ZhYeChWa20}, and the employment of GNNs to learn suitable filtrations have been examined in \cite{HoGrRiNiKw20}. Furthermore, topological techniques have also been used for model analysis in machine learning. For example, topological analysis has been applied to evaluate generative adversarial networks (GANs) by the concept of \emph{Geometry Score} \cite{KhOs18}, while \emph{neural persistence} has been introduced as a complexity measure summarizing topological features that arise when filtrations of the neural network graphs are calculated \cite{RiToBoMoHoGu19}. The topological analysis of the decision boundary of a given classifier has been considered in \cite{RaVaMo19}, and the topological information encoded in the weights of convolutional neural networks (CNNs) has been studied in \cite{GaCa19}.

However, we stress that the development of the theory of GENEOs differs greatly from these lines of research, which are not focused on equivariance concerning arbitrary transformation groups and do not study the topology of suitable operator spaces, but most of them consider the properties of single techniques and applications. In other words, the approach we are interested in is devoted to studying the topological properties of a space of equivariant operators as a whole. In this mathematical setting, the compactification problem can arise and admit resolution.

This paper is structured as follows. In Section \ref{NT}, we introduce basic concepts and give formal definitions. Section \ref{BR} is devoted to summarizing important results on topological groups and GENEOs that will be used in our constructions. Our compactification results are proved in Section \ref{CR}. A brief discussion and an appendix containing some supplementary material conclude the paper.


\section{The Mathematical Setting}
\label{NT}

Let $X$ be a non-empty set and consider the normed vector space $(\mathbb{R}^X_b, \| \cdot \|_\infty)$, where \[\mathbb{R}^X_b=\{\varphi:X\rightarrow\mathbb{R} \mid \varphi \text{ is bounded}\},\] and $\| \cdot \|_\infty$ denotes the usual uniform norm. 
Any metric subspace $(\Phi, D_\Phi)$ of $\mathbb{R}^X_b$, where \[D_\Phi(\varphi_1, \varphi_2) := \| \varphi_1 - \varphi_2 \|_\infty = \sup_{x \in X} \left \vert \varphi_1(x) - \varphi_2(x) \right \vert, \ \mathrm{for\ every\ } \varphi_1, \varphi_2 \in \Phi \] endows $X$ with the topology induced by the extended pseudo-metric
\[ D_X(x_1, x_2) := \sup\limits_{\varphi\in\Phi}\left \vert \varphi(x_1)-\varphi(x_2)\right \vert.\] The space $X$ is interpreted as the space where one makes measurements, and the elements $\varphi$ of $\Phi$ are called \textit{admissible measurements} or \textit{signals}. The function spaces $\Phi$ are sometimes called \textit{data sets}. Moreover, we set $\mathrm{dom}(\Phi) := X$.

In our model, self-maps of the space $X$ have an important role to play.

\begin{definition}
A map $g : X \to X$ is said to be a $\Phi-$\textit{operation} if the composite function $\varphi g$ is an element of $\Phi$ for every $ \varphi \in \Phi$. A bijective $\Phi-$\textit{operation} is called an \textit{invertible} $\Phi-$\textit{operation} if $g^{-1}$ is also a $\Phi-$operation.
\end{definition}

The set of all invertible $\Phi-$operations is denoted by $\mathrm{Aut}_{\Phi}(X)$; i.e., \[ \mathrm{Aut}_{\Phi}(X) := \{ g : X \to X \mid g \ \mathrm{is \ a \ bijection, \ and} \ \varphi g, \ \varphi g^{-1} \in \Phi, \ \mathrm{for \ all} \ \varphi \in \Phi \}, \] and forms a group under the function composition. It acts on the space $\Phi$ through the right action 
\[ \rho : \Phi \times \mathrm{Aut}_{\Phi}(X) \to \Phi, \ (\varphi, g) \mapsto \varphi g. \]

We say that a bijection $f : X \to X$ is an \textit{isometry} of $X$ if $D_X(f(x), f(y)) = D_X(x,y)$, for every $x, y \in X$, and denote the set of all isometries of $X$ by $\mathrm{Iso}(X)$. 

Let $C(X,X) \supseteq \mathrm{Iso}(X)$ denote the set of all continuous functions $f : X \to X$. The following pseudo-metric will be used frequently in the sequel.
 \[ d_\infty(f,g) := \sup_{x \in X} D_X(f(x), g(x)), \ \mathrm{for\ every\ } f, g \in C(X,X). \]  
 If $\Phi$ is rich enough to endow $X$ with a metric structure, instead of a pseudo-metric one, then $d_\infty$ is an extended metric, and is called the \textit{metric of uniform convergence} on $C(X,X)$.

\begin{definition}
If $G$ is a subgroup of $\mathrm{Aut}_{\Phi}(X)$, then $(\Phi, G)$ is called a \textit{perception pair}.
\end{definition}

\begin{definition}
We say that a perception pair $(\Phi, G)$ with $\mathrm{dom}(\Phi) = X$ is \textit{compact} if $\Phi$, $G$, and $X$ are all compact.
\end{definition}

The data set $\Phi$ endows $\mathrm{Aut}_{\Phi}(X)$ with a pseudo-metric structure where the (extended) pseudo-distance $D_{\mathrm{Aut}}$ is given by \[D_{\mathrm{Aut}}(f, g) := \sup_{\varphi \in \Phi} D_{\Phi}(\varphi f, \varphi g), \ \mathrm{for\ every\ } f, g \in \mathrm{Aut}_{\Phi}(X). \]

Conversely, each group $G \subseteq \mathrm{Aut}_{\Phi}(X)$ induces on the space $\Phi$ a pseudo-metric $d_G : \Phi \times \Phi \to \mathbb{R}$: \[ d_G(\varphi_1,\varphi_2) := \inf_{g \in G} D_{\Phi}(\varphi_1, \varphi_2 g), \ \mathrm{for\ every\ } \varphi_1,\varphi_2 \in \Phi. \] 
We call $d_G$ the \textit{natural pseudo-distance} associated with the group $G$. This pseudo-metric represents the ground truth in our model and allows  us to compare functions in the sense that it vanishes for the pairs of functions that are equivalent with respect to the action of the group $G$ representing the data similarities useful for the observer \cite{NatPDistCur, NatPDistSur, NatPDistMan}.

It is known that each invertible $\Phi-$operation is an isometry with respect to $D_X$; that is, $\mathrm{Aut}_{\Phi}(X) \subseteq \mathrm{Iso}(X)$ \cite[Proposition 2]{MaIntPa}. But $d_{\infty}$ does not endow the space $(\mathrm{Aut}_{\Phi}(X), D_{\mathrm{Aut}})$ with any additional pseudo-metric structure:
\begin{align*}
    D_{\mathrm{Aut}}(f, g) & := \sup_{\varphi \in \Phi} D_{\Phi}(\varphi f, \varphi g) \\
    & = \sup_{\varphi \in \Phi} \sup_{x \in X} \vert \varphi f(x) - \varphi g(x) \vert \\
    & = \sup_{x \in X} D_X (f(x), g(x)) \\
    & =: d_\infty (f, g),
\end{align*}
for all $f, g \in \mathrm{Aut}_{\Phi}(X).$ So, $d_\infty$ coincides with the pseudo-distance $D_{\mathrm{Aut}}$ on $\mathrm{Aut}_{\Phi}(X)$; that is \[ d_\infty \vert_{\mathrm{Aut}_{\Phi}(X)} = D_{\mathrm{Aut}}. \]

In general, $D_{\mathrm{Aut}}$ is an extended pseudo-metric. But when $(X, D_X)$ is a metric space, then so is $(G, D_{\mathrm{Aut}})$: If $g, h \in G$ are distinct functions, then there is an $x_0 \in X$ such that $g(x_0) \neq h(x_0)$. Since $D_X$ is a metric, \[ 0 < D_X(g(x_0), h(x_0)) \leq \sup_{x \in X} D_X(g(x), h(x)) = d_\infty(g,h) = D_\mathrm{Aut}(g,h), \] whence $D_\mathrm{Aut}$ is a metric as well.

\begin{definition}
Let $(\Phi, G)$ and $(\Psi, H)$ be perception pairs with $\mathrm{dom}(\Phi) = X$ and $\mathrm{dom}(\Psi) = Y$, and $T : G \to H$ be a group homomorphism.  A map $F : \Phi \to \Psi$ is said to be a \textit{group equivariant non-expansive operator} (GENEO) \textit{with respect to $T$} if \[ F(\varphi \circ g) = F(\varphi) \circ T(g), \  \mathrm{for\ every\ } \varphi \in \Phi, g \in G, \] and \[ \| F(\varphi_1) - F(\varphi_2) \|_\infty \leq \|\varphi_1 - \varphi_2\|_\infty, \ \mathrm{for\ every\ } \varphi_1, \varphi_2 \in \Phi. \]
\end{definition}

A map $F : \Phi \to \Psi$ satisfying the first condition is called $T-$\textit{equivariant} or a \textit{group equivariant operator} (GEO), and it is called \textit{non-expansive} if it satisfies the second condition.
For the sake of conciseness, we often write a GENEO as $(F, T) : (\Phi, G) \to (\Psi, H)$.

The set $\mathcal{F}_T^{\mathrm{all}}$ of all GENEOs $(F, T) : (\Phi, G) \to (\Psi, H)$ corresponding to a group homomorphism $T : G \to H$ is a metric space with the distance function given by \[ D_{\mathrm{GENEO}}(F_1, F_2) = \sup_{\varphi \in \Phi} D_{\Psi}(F_1(\varphi), F_2(\varphi)), \ \mathrm{for\ every\ } F_1, F_2 \in \mathcal{F}_T^{\mathrm{all}}. \]

The natural pseudo-distance allows us to define another pseudo-metric on this space: \[ D_{\mathrm{GENEO}, H}(F_1, F_2) := \sup_{\varphi \in \Phi} d_{H} (F_1(\varphi), F_2(\varphi)), \ \mathrm{for\ every\ } F_1, F_2 \in \mathcal{F}_T^{\mathrm{all}}. \]

The spaces  $\mathcal{F} \subseteq \mathcal{F}_T^{\mathrm{all}}$ of GENEOs prove instrumental in comparing data. For example, one can consider the following pseudo-metric: \[ D_{\mathcal{F}, \Phi}(\varphi_1,\varphi_2):= \sup_{F \in \mathcal{F}} D_{\Psi}(F(\varphi_1),F(\varphi_2)), \ \mathrm{for\ every\ } \varphi_1, \varphi_2 \in \Phi. \]

Conti et al. (2022) \cite{ConstGENEOSym} give examples demonstrating how the use of GENEOs increases our ability to distinguish between data.

Our objective is to obtain isometric embeddings of perception pairs and of the spaces of GENEOs into compact ones while retaining the metric properties of the original spaces.
The reader is referred to \cite{MaIntPa, NQTh} for further details about the concepts we have so far introduced in this section.


We will assume in Section \ref{CR} that the data set $\Phi$ is rich enough to endow the common domain $X$ with a metric structure. The first step towards constructing our compactifications, under this assumption, is to consider the metric completion of $X$. It is well known that every metric space $(M, D_M)$ admits a unique metric completion $(\hat{M}, \hat{D}_{\hat{M}})$ up to homeomorphisms. We can assume that the completion $\hat{M}$ contains $M$; i.e., we have the inclusion \[ j : M \to \hat{M}, \] and the metric $\hat{D}_{\hat{M}}$ is given by
\[\hat{D}_{\hat{M}}(\hat{x},\hat{y})= \lim_{n \to \infty} D_M(x_n,y_n),\]
where $\hat{x},\hat{y} \in \hat{M}$, and $(x_n)_{n\in \mathbb{N}}$, $(y_n)_{n\in \mathbb{N}}$ are arbitrary sequences in $M$ converging to $\hat{x}$ and $\hat{y}$ respectively.


\section{Basic results on topological groups and group equivariant non-expansive operators}
\label{BR}

We recall the following results {\color{black}from \cite{MaIntPa, Munk, NQTh}} which will be used frequently in the sequel. The proofs of the results that appear only in \cite{NQTh} will be given in Section \ref{App} for the sake of completeness.

\begin{proposition}\label{SigNonExp}
\cite[Proposition 1.2.10]{NQTh} Each function $\varphi \in \Phi$ is non-expansive, and hence uniformly continuous with respect to $D_X$.
\end{proposition}

Therefore, the topology $\tau_{D_X}$ induced by $D_X$ is finer than the initial topology $\tau_{\text{in}}$ on $X$, which is the coarsest topology on $X$ with respect to which all the signals $\varphi \in \Phi$ are continuous.

\begin{theorem}\label{InitTop}
\cite[Supplementary Methods: Theorem 2.1]{MaIntPa} If $\Phi$ is totally bounded, then $\tau_{D_X}$ coincides with $\tau_{\text{in}}$.
\end{theorem}

\begin{theorem}\label{PhiBddXBdd}
\cite[Theorem 1]{NQTh} If $\Phi$ is totally bounded, then so is $(X,D_X)$.
\end{theorem}

\begin{proposition}\label{PhiOpIso}
\cite[Proposition 2]{MaIntPa} $\mathrm{Aut}_{\Phi}(X) \subseteq \mathrm{Iso}(X).$
\end{proposition}

That is, each $g \in G \subseteq \mathrm{Aut}_{\Phi}(X)$ is an isometry of $X$.

Recall that a subgroup of a topological group is topological, and

\begin{proposition}\label{TpSbGp}
\cite{Munk} If $A$ is a subgroup of a topological group $G$, then $cl_G(A)$ is also a subgroup, and hence a topological subgroup of $G$.
\end{proposition}

This proposition will be used in conjunction with

\begin{theorem}\label{AutTopGp}
\cite[Supplementary Methods: Theorem 2.7]{MaIntPa} $\mathrm{Aut}_{\Phi}(X)$ is a topological group and the action $\rho : \Phi \times \mathrm{Aut}_{\Phi}(X) \to \Phi$ is continuous.
\end{theorem}

\begin{theorem}\label{PhiBddGBdd}
\cite[Theorem 4]{NQTh} If $\Phi$ is totally bounded, then so is $(G, D_{\mathrm{Aut}})$.
\end{theorem}

\begin{proposition}\label{IsoXComp}
\cite[Proposition 1.2.20]{NQTh} If $(X, D_X)$ is a compact metric space, then $(\mathrm{Iso}(X), d_\infty)$ is also compact.
\end{proposition}

\begin{theorem}\label{AutComp}
\cite[Theorem 5]{NQTh} If $\Phi \subseteq \mathbb{R}^X_b$ and $(X, D_X)$ are both compact metric spaces, then $\mathrm{Aut}_{\Phi}(X)$ is closed in $\mathrm{Iso}(X)$, and hence compact.
\end{theorem}

\begin{theorem}\label{SpGNComp}
\cite[Theorem 7]{MaIntPa} The space $(\mathcal{F}_T^{\mathrm{all}}, D_{\mathrm{GENEO}})$ of GENEOs $(F, T) : (\Phi, G) \to (\Psi, H)$ is compact whenever the spaces $\Phi$ and $\Psi$ are compact.
\end{theorem}


\section{Our compactification results}
\label{CR}

Let $(\Phi, G)$, $\mathrm{dom}(\Phi)=X$ and $(\Psi, H)$, $\mathrm{dom}(\Psi)=Y$ be perception pairs and $\mathcal{F} \subseteq \mathcal{F}^{\mathrm{all}}_{T}$, where $\mathcal{F}_T^{\mathrm{all}}$ denotes, as usual, the space of all GENEOs $(F, T) : (\Phi, G) \to (\Psi, H)$ with respect to a fixed homomorphism $T : G \to H$. 

In this section we will be assuming that
\begin{description}
  \item[$i)$] $\Phi$ and $\Psi$ are totally bounded, and are rich enough to endow each of $X$ and $Y$ with metric structures;
  \item[$ii)$] the collection of sets $\{F(\Phi) \mid F \in \mathcal{F}\}$ covers $\Psi$.
\end{description}
We know that even if $\Phi$ and $\Psi$ are compact, let alone being totally bounded, $X, G, Y$, and $H$ need not be compact \cite{MaIntPa}, though $\mathcal{F}_T^{\mathrm{all}}$ is indeed compact in that case. Moreover, an arbitrary subspace $\mathcal{F}$ of $\mathcal{F}_T^{\mathrm{all}}$ need not necessarily be compact either. Since compactness is an important property, as it provides us with essential guarantees in machine learning context, it is natural to prefer compact spaces in practical applications. We therefore ask: If compactness of $X, G, Y$, and $H$ is not guaranteed even by the compactness of data sets $\Phi$ and $\Psi$, let alone their total boundedness, can we at least prove that these spaces can be isometrically and densely embedded in compact ones while the corresponding sought after compact spaces preserve the former mutual relations between the original spaces? That is, can we find compactifications of perception pairs? Furthermore, can we find compactifications of the spaces of GENEOs? These notions need being made precise, which we do in the sequel, and prove that our assumptions are sufficient to grant the answer to this question in the affirmative.

Somewhat formally, \textit{given} the perception pairs $(\Phi, G)$, $\mathrm{dom}(\Phi)=X$ and $(\Psi, H)$, $\mathrm{dom}(\Psi)=Y$ and a space $\mathcal{F} \subseteq \mathcal{F}_T^{\mathrm{all}}$ of GENEOs $(F, T) : (\Phi, G) \to (\Psi, H)$ with respect to a fixed homomorphism $T : G \to H$, we assume that the data sets $\Phi$ and $\Psi$ are totally bounded and rich enough to endow $X$ and $Y$ with metric structures, and the collection $\{F(\Phi) \mid F \in \mathcal{F} \}$ covers the space $\Psi$. Under these assumptions, we \textit{find} perception pairs $(\Phi^*, G^*)$, $\mathrm{dom}(\Phi^*)=X^*$ and $(\Psi^*, H^*)$, $\mathrm{dom}(\Psi^*)=Y^*$, a space ${\mathcal{F}}^* \subseteq \mathcal{F}_{T^*}^{\mathrm{all}}$ of GENEOs $(F^*, T^*) : (\Phi^*, G^*) \to (\Psi^*, H^*)$ with respect to a fixed homomorphism $T^* : G^* \to H^*$, and isometric embeddings $j_1 : X \to X^*$, $j_2 : Y \to Y^*$, $i_1 : \Phi \to \Phi^*$, $i_2 : \Psi \to \Psi^*$, $k_1: G \to G^*$, $k_2 : H \to H^*$, and $f : \mathcal{F} \to {\mathcal{F}}^*$. We \textit{require} that the spaces $\Phi^*, G^*, X^*, \Psi^*, H^*, Y^*$ and ${\mathcal{F}}^*$ are all compact, and the following commutativity conditions are satisfied:
$i_1(\varphi) \circ j_1 = \varphi$ for every $\varphi \in \Phi$, $i_2(\psi) \circ j_2 = \psi$ for every $\psi \in \Psi$; $k_1(g) \circ j_1 = j_1 \circ g$ for every $g \in G$, $k_2(h) \circ j_2 = j_2 \circ h$ for every $h \in H$; $i_2 \circ F = f(F) \circ i_1$ for every $F \in \mathcal{F}$; and $k_2 \circ T = T^* \circ k_1$. 

These compatibility conditions formalize the requirement that the spaces $\Phi, G, X, \Psi, H, Y$ and ${\mathcal{F}}$ do not lose any of their metric or topological properties while being viewed as subspaces of $\Phi^*, G^*, X^*, \Psi^*, H^*, Y^*$ and ${\mathcal{F}}^*$ respectively. In this case, we say that $(\Phi^*, G^*)$, $\mathrm{dom}(\Phi^*)=X^*$ is a compactification of the perception pair $(\Phi, G)$, $\mathrm{dom}(\Phi)=X$, and $\mathcal{F}^*$ is a compactification of the space $\mathcal{F}$ of GENEOs. We will give formal definitions in the forthcoming subsections. Our assumptions here are mild; in many practical applications, they are already satisfied.

Precisely, the intermediary results and constructions in Subsections 4.1 and 4.2 are aimed at proving that every perception pair $(\Phi, G)$, $\mathrm{dom}(\Phi) = X$, with totally bounded $\Phi$ endowing $X$ with a metric structure, admits a compactification $(\Phi^*, G^*)$, $\mathrm{dom}(\Phi^*) = X^*$. Similarly, the Subsections 4.3 and 4.4 are devoted to proving that every space $\mathcal{F} \subseteq \mathcal{F}_{T}^{\mathrm{all}}$ of GENEOs $(F, T) : (\Phi, G) \to (\Psi, H)$ with $\mathrm{dom}(\Phi) = X$ and $\mathrm{dom}(\Psi) = Y$ such that the collection $\{ F(\Phi) \mid F \in \mathcal{F} \}$ covers $\Psi$ admits a compactification $\mathcal{F}^*$, provided the data sets $\Phi$ and $\Psi$ are totally bounded and endow $X$ and $Y$ with metric structures. Again, this proof will require several auxiliary constructions and corresponding results.

In order to set the stage for the requisite compactification of the perception pair $(\Phi, G)$, $\mathrm{dom}(\Phi)=X$, we consider the unique metric completion $\hat{X}$ of $X$, and assume that $X \subseteq \hat{X}$. Since $X$ is totally bounded by Theorem \ref{PhiBddXBdd}, $\hat{X}$ is totally bounded by virtue of the isometric embedding $j : X \to \hat{X}$, and hence compact. This serves as the sought after $X^*$ in our construction. Then we use the measurements $\varphi \in \Phi$ and isometries $g \in G$ to define the  measurements $\hat{\varphi} : \Phi \to \mathbb{R}$ on the compact space $\hat{X}$ and its corresponding isometries $\hat{g} : \hat{X} \to \hat{X}$.


\subsection{The Extension of Signals}
\label{ExSg}

It can easily be proved that

\begin{proposition}
Let  $(M,d_M)$ be a metric space, and $S$ a subset of $M$. Then every non-expansive map $f \colon S \to \mathbb{R}$ admits a unique non-expansive extension $\bar{f} \colon \overline{S} \to \mathbb{R}$.
\end{proposition}

Since each $\varphi \in \Phi$ is non-expansive by Proposition \ref{SigNonExp}, we have

\begin{corollary}\label{ExtNonExp}
Each signal $\varphi \in \Phi$ can be uniquely extended to a non-expansive signal $\hat{\varphi}\colon \hat{X} \to \mathbb{R}$, where $\hat X$ is the completion of $X=\mathrm{dom}(\Phi)$, by setting \[ \hat{\varphi}(\hat{x}) = \lim_{n \to \infty} \varphi(x_n), \] for any arbitrary sequence $(x_n)_{n \in \mathbb{N}}$ in $X$ that converges to $\hat{x} \in \hat{X}$.
\end{corollary}

Let us put \[ \hat{\Phi} := \{ \hat{\varphi} : \hat{X} \to \mathbb{R} \mid \varphi \in \Phi \}.\]

Since the extensions $\hat{\varphi}\colon \hat{X} \to \mathbb{R}$ of signals $\varphi \in \Phi$ are unique, we get a one-to-one correspondence $i\colon \Phi \to \hat{\Phi}$ between signals in $\Phi$ and signals in $\hat{\Phi}$ given by \[ \varphi \mapsto \hat{\varphi}. \]

The notations $\mathbb{R}^{\hat{X}}_b$, $\mathrm{Iso}(\hat{X})$, $\hat{d}_\infty$ and $\mathrm{Aut}_{\hat{\Phi}}(\hat{X})$ are self-explanatory. Clearly, $\hat{\Phi} \subseteq \mathbb{R}^{\hat{X}}_b$.

The set $\hat{\Phi}$ of extended signals induces the pseudo-metric $D_{\hat{X}}$ on the completion $\hat{X}$ given by

\[ D_{\hat{X}}(x, y) := \sup_{\varphi \in \hat{\Phi}} \left \vert \varphi(x) - \varphi(y) \right \vert, \quad x, y \in \hat{X}. \]

At this point, the completion $\hat{X}$ appears to be equipped with the previously defined metric $\hat{D}_{\hat{X}}$ associated with the completion, and the pseudo-metric $D_{\hat{X}}$. It is worth investigating their relationship. We will prove later in this subsection that these seemingly distinct functions are in fact numerically equal on $\hat{X}$, thereby establishing in addition that $D_{\hat{X}}$ is in fact a metric.

Before proceeding, we record another general proposition, omitting the easy proof, which will be used frequently in the paper:
\begin{proposition}\label{SpDns}
Let $K$ be a compact topological space and $A$ be dense in $K$. If $f$ is a continuous real-valued function on $K$, then 
\[\sup f(K) = \sup f(A).\]
\end{proposition}

We are ready now to prove the following theorem:
\begin{theorem}\label{ExtSgIso}
The correspondence $i \colon \Phi \to  \hat{\Phi}$ is an isometry.
\end{theorem}

\begin{proof}
The map $i$ is surjective by construction; it will suffice to prove that it preserves distances, i.e.,
\[\| i(\varphi_1)- i (\varphi_2)\|_\infty = \| \varphi_1 - \varphi_2 \|_\infty,\] for any $\varphi_1, \varphi_2 \in \Phi$. Since $X$ is dense in the compact topological space $\hat{X}$, and $i(\varphi) = \hat{\varphi}$ is an extension of $\varphi \in \Phi$, by Proposition \ref{SpDns} we have
\begin{align*}
   \| i(\varphi_1) - i (\varphi_2) \|_\infty & := \sup_{x \in \hat{X}} \vert \widehat{\varphi_1}(x) - \widehat{\varphi_2}(x) \vert \\
   & = \sup_{x \in X} \vert \widehat{\varphi_1}(x) - \widehat{\varphi_2}(x) \vert \\
   & = \sup_{x \in X} \vert \varphi_1(x)- \varphi_2(x) \vert \\
   & =: \| \varphi_1 - \varphi_2 \|_\infty,
\end{align*}
for any $\varphi_1, \varphi_2 \in \Phi$.
Therefore, $i$ is an isometry.
\end{proof}

\begin{corollary}\label{CpPhiTBdd}
The space $\hat{\Phi}$ of extended signals is totally bounded.
\end{corollary}

\begin{remark}\label{ComSg}
Let us consider the isometry $i:\Phi\to\hat\Phi$ and the inclusion $j:X\hookrightarrow\hat X$. Since $\hat{\varphi}$ extends $\varphi$, the following natural commutativity condition holds for each $\varphi \in \Phi$: \[ i(\varphi) \circ j = \varphi. \]
\end{remark}

The total boundedness of $\hat{\Phi}$ allows us to prove the following crucial statement.

\begin{proposition}\label{PsMetEq}
On $\hat{X}$, $D_{\hat{X}}=\hat{D}_{\hat{X}}$.
\end{proposition}

\begin{proof}
Since $\hat{\Phi}$ is totally bounded, $D_{\hat{X}}$ induces on $\hat{X}$ the initial topology with respect to $\hat{\Phi}$ by Theorem \ref{InitTop}. Moreover, by Corollary \ref{ExtNonExp} the functions in $\hat{\Phi}$ are continuous with respect to $\hat{D}_{\hat{X}}$ as well. Hence, the topology induced by $\hat{D}_{\hat{X}}$ is finer than the topology induced by $D_{\hat{X}}$. This directly implies that $D_{\hat{X}}$ is a continuous function with respect to $\hat{D}_{\hat{X}}$. Then, we have
\[\hat{D}_{\hat{X}}(\hat{x},\hat{y}):= 
\lim_{n \to \infty} D_X(x_n,y_n) = \lim_{n \to\infty} D_{\hat{X}}(x_n,y_n)= D_{\hat{X}}(\hat{x},\hat{y})\] 
where $\hat{x},\hat{y} \in \hat{X}$, and $(x_n)_{n \in \mathbb{N}},(y_n)_{n \in \mathbb{N}}$ are sequences in $X$ converging to $\hat{x}$ and $\hat{y}$ respectively, with reference to the topology induced by $\hat{D}_{\hat{X}}$.
\end{proof}

So, $D_{\hat{X}}$ is a metric. As pointed out in Section \ref{NT}, this directly implies that the pseudo-metric $\hat{D}_{\mathrm{Aut}}$ induced by $\hat{\Phi}$ on $\mathrm{Aut}_{\hat{\Phi}}(\hat{X})$ is also a metric.


\subsection{The Isometries of the Completions}
\label{IsoComp}

We now turn to the construction of an auxiliary topological group $\hat{G} \subseteq \mathrm{Aut}_{\hat{\Phi}}(\hat{X})$, isometric to the given group $G \subseteq \mathrm{Aut}_{\Phi}(X)$, whose closure $G^* = \overline{\hat{G}}$ in the compact space $\mathrm{Iso}(\hat{X})$, finally, serves our purposes.


\subsubsection{The Induced Bijections}
\label{CpG}

Each $ g \in G \subseteq \mathrm{Aut}_{\Phi}(X) \subseteq \mathrm{Iso}(X)$ induces a self-map $ \hat{g} : \hat{X} \to \hat{X} $ on the metric completion $\hat{X}$ by the following association: \[ \hat{g}(\hat{x}) := \lim_{n \to \infty} g(x_n), \] where $(x_n)_{n \in \mathbb{N}}$ is a sequence in $X$ converging to $\hat{x} \in \hat{X}$. The sequence $(g(x_n))_{n \in \mathbb{N}}$ is a Cauchy sequence since $g$ is an isometry by Proposition \ref{PhiOpIso}. Also, if $(y_n)_{n \in \mathbb{N}}$ is another sequence in $X$ converging to $\hat{x}$, then as $g$ is an isometry, we have \[ 0 = \lim_{n \to \infty} D_X(x_n, y_n) = \lim_{n \to \infty} D_X(g(x_n), g(y_n)), \] whence \[g(y_n) \rightarrow \lim_{n \to \infty} g(x_n)\] as well. So, the function $\hat{x} \mapsto \lim_{n \to \infty} g(x_n)$ is well defined. 

Moreover, note that $\hat{g} \vert_X = g$.

\begin{proposition}\label{IndBjInv}
The map $\hat{g} : \hat{X} \to \hat{X}$ is bijective for every $g \in G$, and  
$\hat{g}^{-1}=\widehat{g^{-1}}$.
\end{proposition}
\begin{proof}
Let $\hat{x}_1, \hat{x}_2 \in \hat{X}$ with $\hat{x}_1 \neq \hat{x}_2$, and $(x_{1, n})_{n \in \mathbb{N}}$ and $(x_{2, n})_{n \in \mathbb{N}}$ be sequences in $X$ converging respectively to $\hat{x}_1$ and $ \hat{x}_2$. As $g$ is an isometry, 
\begin{align*}
    0 \neq \hat{D}_{\hat{X}}(\hat{x}_1, \hat{x}_2)
    & := \lim_{n \to \infty} D_X(x_{1,n}, x_{2,n}) \\
    & = \lim_{n \to \infty} D_X(g(x_{1,n}), g(x_{2,n})) \\
    & =: \hat{D}_{\hat{X}}(\hat{g}(\hat{x}_1), \hat{g}(\hat{x}_2)),
\end{align*}
whence $\hat{g}(\hat{x}_1) \neq \hat{g}(\hat{x}_2)$ and $\hat{g}$ is injective.

As for surjectivity, let $g \in G$. If $\hat{y}\in\hat{X}$, then there is a sequence $(y_n)_{n \in \mathbb{N}}$ in $X$ such that $y_n \rightarrow \hat y$ in $\hat X$.
As $g^{-1}$ exists, we can put $x_n := g^{-1}(y_n)$, for each $ n \in \mathbb{N}$. 
Since $g^{-1}$ is an isometry 
and $(y_{n})_{n \in \mathbb{N}}$ is a Cauchy sequence, $(x_{n})_{n \in \mathbb{N}}$ too is a Cauchy sequence; so it converges to some $\hat{x} \in \hat{X}$. Of course, $\hat g(\hat x):=\lim_{n \to \infty} g(x_n) = \lim_{n \to \infty} y_n =\hat y$, and hence $\hat g$ is surjective.

Also, the equality $\hat g(\hat x) = \hat y$ just proved can be rewritten as \[ \hat{g}^{-1}(\hat{y}) = \hat{x}. \] But at the same time, as $g^{-1} \in G$, by definition we have \[ \widehat{g^{-1}}(\hat{y}) = \widehat{g^{-1}}(\lim_{n \to \infty} y_n) = \lim_{n \to \infty} g^{-1}(y_n) = \lim_{n \to \infty} x_n = \hat{x}. \] By the arbitrariness of $\hat{y} \in \hat{X}$, we get \[ \hat{g}^{-1}=\widehat{g^{-1}}. \]
\end{proof}

The following important property will be used frequently in the sequel.

\begin{proposition}\label{IndBjCmt}
For each $\varphi \in \Phi$ and each $g \in G$, \[ \hat{\varphi} \hat{g} = \widehat{\varphi g}. \]
\end{proposition}

\begin{proof}
As $g \in \mathrm{Aut}_{\Phi}(X)$, $\varphi g \in \Phi$. So, if $\hat{x} \in \hat{X}$ and $(x_n)_{n \in \mathbb{N}}$ is a sequence in $X$ converging to $\hat{x}$, we compute:
\begin{align*}
    \hat{\varphi}\hat{g}(\hat{x}) & = \hat{\varphi}(\hat{g}(\hat{x})) \\
    & = \hat{\varphi}(\lim_{n \to \infty} g(x_n)) \\
    & = \lim_{n \to \infty} \varphi(g(x_n)) \\
    & = \lim_{n \to \infty} \varphi g(x_n) \\
    & = \widehat{\varphi g}(\lim_{n \to \infty} x_n) \\
    & = \widehat{\varphi g}(\hat{x}).
\end{align*}
By the arbitrariness of $\hat{x}$, we have the proposed equality.
\end{proof}

\begin{corollary}\label{IndBjInvPhiOp}
For each $g \in G$, $ \hat{g} \in \mathrm{Aut}_{\hat{\Phi}}(\hat{X}) \subseteq \mathrm{Iso}(\hat{X})$.
\end{corollary}
\begin{proof}
Let $\hat{\varphi} \in \hat{\Phi}$. 
As $g \in G \subseteq \mathrm{Aut}_{\Phi}(X)$, $\varphi g \in \Phi$; so, $\hat{\varphi} \hat{g} = \widehat{\varphi g} \in \hat{\Phi}$ by Proposition \ref{IndBjCmt}; whence $\hat{g}$ is a $\hat{\Phi}-$operation.
Since $\hat{g}^{-1} = \widehat{g^{-1}}$ (Proposition \ref{IndBjInv}), by applying Proposition \ref{IndBjCmt} again to $g^{-1}\in G$, we infer that $\hat{g}^{-1}$ is a $\hat{\Phi}-$operation too; whence $\hat{g} \in \mathrm{Aut}_{\hat{\Phi}}(\hat{X})$. 

The inclusion $\mathrm{Aut}_{\hat{\Phi}}(\hat{X}) \subseteq \mathrm{Iso}(\hat{X})$ is stated in Proposition \ref{PhiOpIso}.
\end{proof}

\begin{proposition}\label{IndBjFnCompCmt}
For each $g, h \in G$, \[ \hat{g}\hat{h} = \widehat{gh}. \]
\end{proposition}

\begin{proof}
Let $\hat{x} \in \hat{X}$, and $(x_n)_{n \in \mathbb{N}}$ be a sequence in $X$ converging to $\hat{x}$ in $\hat X$. Then, by recalling the definitions of $\hat h$ and $\hat g$, we have
\begin{align*}
    \hat{g}\hat{h}(\hat{x}) & = \hat{g}(\hat{h}(\hat{x})) \\
    & = \hat{g}(\lim_{n \to \infty} h(x_n) ) \\
    & = \lim_{n \to \infty} g(h(x_n)) \\
    & = \lim_{n \to \infty} gh(x_n) \\
    & = \widehat{gh}(\hat{x}),
\end{align*}
whence by the arbitrariness of $\hat{x}$, the proposition is proved.
\end{proof}

Let us put \[ \hat G:=\{\hat g : \hat{X} \to \hat{X} \mid g\in G\} \].

\begin{remark}
Clearly, $\widehat{\mathrm{id}_X} = \mathrm{id}_{\hat{X}}\in \hat{G}$.
\end{remark}

\begin{corollary}\label{IndBjGp}
The set $\hat{G}$ is a subgroup of $\mathrm{Aut}_{\hat{\Phi}}(\hat{X})$.
\end{corollary}
\begin{proof}
It will suffice to show that $\hat{G}$ is closed under composition and computation of the inverse. The first property follows from Proposition \ref{IndBjFnCompCmt}, since if $\hat g, \hat h \in \hat G$, then $\hat{g}\hat{h} = \widehat{gh} \in \hat{G}$. The second property follows from Proposition
\ref{IndBjInv}, since if $\hat g\in \hat G$, then $\hat{g}^{-1}=\widehat{g^{-1}}\in \hat G$.
\end{proof}

\begin{remark}
Corollary \ref{IndBjGp} implicitly states that $(\hat{\Phi}, \hat{G})$ is a perception pair.
\end{remark}

Note that $\mathrm{Aut}_{\hat{\Phi}}(\hat{X})$, and therefore $\hat{G} \subseteq \mathrm{Aut}_{\hat{\Phi}}(\hat{X})$, are pseudo-metric spaces with the pseudo-metric $\hat{D}_{\mathrm{Aut}} : \mathrm{Aut}_{\hat{\Phi}}(\hat{X}) \times \mathrm{Aut}_{\hat{\Phi}}(\hat{X}) \to \mathbb{R}$ given by \[ \hat{D}_{\mathrm{Aut}}(\hat{g}, \hat{h}) := \sup_{\hat{\varphi} \in \hat{\Phi}} D_{\hat{\Phi}}(\hat{\varphi}\hat{g}, \hat{\varphi}\hat{h}), \ \mathrm{for \ every} \ \hat{g}, \hat{h} \in \mathrm{Aut}_{\hat{\Phi}}(\hat{X}). \] 
Moreover,
\begin{align*}
    \hat{D}_{\mathrm{Aut}}(\hat{g}, \hat{h}) & := \sup_{\hat{\varphi} \in \hat{\Phi}} D_{\hat{\Phi}}(\hat{\varphi}\hat{g}, \hat{\varphi}\hat{h}) \\
   & = \sup_{\hat{\varphi} \in \hat{\Phi}} \sup_{\hat{x} \in \hat{X}} \left \vert \hat{\varphi}(\hat{g}(\hat{x})) - \hat{\varphi}(\hat{h}(\hat{x})) \right \vert \\
   & = \sup_{\hat{x} \in \hat{X}} \sup_{\hat{\varphi} \in \hat{\Phi}} \left \vert \hat{\varphi}(\hat{g}(\hat{x})) - \hat{\varphi}(\hat{h}(\hat{x})) \right \vert \\
   & = \sup_{\hat{x} \in \hat{X}} D_{\hat{X}}(\hat{g}(\hat{x}), \hat{h}(\hat{x})).
\end{align*}

Therefore, if $\hat{g}, \hat{h}\in\mathrm{Aut}_{\hat{\Phi}}(\hat{X})$ and $\hat{D}_{\mathrm{Aut}}(\hat{g}, \hat{h})=0$, then $D_{\hat{X}}(\hat{g}(\hat{x}), \hat{h}(\hat{x}))=0$ for every $\hat x\in \hat X$.
Since $\hat{\Phi}$ endows $\hat{X}$ with the metric structure induced by the coinciding metrics $D_{\hat X}$ and $\hat D_{\hat X}$ (Proposition \ref{PsMetEq}), $\hat{g}(\hat{x})=\hat{h}(\hat{x})$ for every $\hat x\in \hat X$, and hence $\hat{g}=\hat{h}$. It follows that $\hat{G}$, and therefore $\hat{G} \subseteq \mathrm{Aut}_{\hat{\Phi}}(\hat{X})$, are metric spaces.

\begin{proposition}\label{kiso}
The correspondence $k : G \to \hat{G}$ given by $k(g) := \hat{g}$ is an isometry.
\end{proposition}
\begin{proof}
The map $k$ is injective: If $g, h \in G$ differ at some $ x \in X $, $\hat{g}(x) \neq \hat{h}(x)$ as well, and $k(g) = \hat{g} \neq \hat{h} = k(h)$. Also, the definition of $\hat G$ immediately implies that $k$ is surjective.

We show that $k$ preserves distances. By Corollary \ref{IndBjInvPhiOp}, the real-valued function $f: \hat{X} \to \mathbb{R}$ defined by setting \[ f(\hat{x}) := D_{\hat{X}}(\hat{g}(\hat{x}), \hat{h}(\hat{x})), \ \mathrm{for \ every} \ \hat{x} \in \hat{X}, \] is continuous for every $\hat g$ and $\hat h$ in $\hat G$, since each isometry is by definition a continuous map.

Let $\hat{g}, \hat{h} \in \hat{G}$; then by Propositions \ref{SpDns}, we have
\begin{align*}
    \hat{D}_{\mathrm{Aut}}(\hat{g}, \hat{h}) 
    & = \sup_{\hat{x} \in \hat{X}} D_{\hat{X}}(\hat{g}(\hat{x}), \hat{h}(\hat{x})) \\
    & = \sup_{x \in X} D_{\hat{X}}(\hat{g}(x), \hat{h}(x)) \\
    & = \sup_{x \in X} D_X( g(x), h(x)) \\
    & =: D_{\mathrm{Aut}}(g, h),
\end{align*}
as required.
\end{proof}

\begin{corollary}
If $G$ is complete, then $\hat{G}$ is compact.
\end{corollary}
\begin{proof}
Recall that the space $\Phi$ of admissible signals was assumed to be totally bounded; whence by Theorem \ref{PhiBddGBdd}, $G$ is totally bounded, and being complete by hypothesis, it is compact. As $k : G \to \hat{G}$ is an isometry, $\hat{G}$ is compact as well.
\end{proof}

\begin{remark}\label{remGnotcomplete}
The assumption here that $G$ is complete cannot be removed. It is easy to give an example of a perception pair $(\Phi, G)$ where $\Phi$ is compact but $G$ is not complete \cite{MaIntPa}.
For example, if $\Phi$ is the compact space of all $1$-Lipschitz functions from $X=S^1=\{(x,y)\in\mathbb{R}^2:x^2+y^2=1\}$ to $[0,1]$, and $G$ is the group of all rotations $\rho_{2\pi q}$ of $X$ of $2\pi q$ radians with $q$ a rational number, then the topological group $G$ is not complete. Moreover, in this case $\hat X=X$, and the topological group $\hat G=G$ is not compact either.
\end{remark}


It is easy to see that the embeddings $j: X \to \hat{X}$ and $k : G \to \hat{G}$ satisfy the following natural commutativity condition.

\begin{proposition}\label{ComBj}
For each $g \in G$, \[ k(g) \circ j = j \circ g. \] That is, for each $g \in G$ and each $x \in X$, \[ \hat{g}(x)=\widehat{g(x)}. \]
\end{proposition}

\begin{remark}
We observe that $\hat{g}$ is the only map in $\mathrm{Aut}_{\hat{\Phi}}(\hat{X})$ with $\hat{g} \vert_X = g$. It is indeed easy to show that for any
$\overline{g}\in\mathrm{Aut}_{\hat{\Phi}}(\hat{X})$
such that $\overline g\vert_X = g$, the equality
$\overline{g}(\hat x) = \hat{g}(\hat x)$ holds for every $\hat x\in\hat X$.
\end{remark}

Before proceeding, we stress that while $\hat X$, by definition, is a complete topological space, the topological spaces $\hat \Phi$ and $\hat G$, in general, are not complete.


\subsubsection{The Embedding of Perception Pairs}\label{subSectEPPT}

We are now ready to show that every perception pair $(\Phi, G)$, $\mathrm{dom}(\Phi) = X$ can be embedded in a compact perception pair $(\Phi^*, G^*)$, $\mathrm{dom}(\Phi^*) = X^*$, provided that the space $\Phi$ of signals is totally bounded and $(X, D_X)$ is a metric space.

We have so far obtained only an isometric image $\hat{G}$ of $G$. The group $G$ is chosen arbitrarily; so $G$ and $\hat{G}$ may or may not be closed in $\mathrm{Aut}_{\Phi}(X)$ and $\widehat{\mathrm{Aut}_{\Phi}(X)}$ respectively. Similarly, the space $\hat{\Phi}$, being isometric to $\Phi$, need not necessarily be compact.
However, the space $C(\hat{X}, \mathbb{R)}$ is complete; so, $\overline{\hat{\Phi}}$, the closure of $\hat{\Phi}$ in $C(\hat{X}, \mathbb{R)}$, being closed, is complete as well. Also, $\hat{\Phi}$, being isometric to the totally bounded space $\Phi$ (Theorem \ref{ExtSgIso}), is totally bounded, and so is its closure. Consequently, 
\begin{proposition}\label{ClCpPhComp}
The metric space $\overline{\hat{\Phi}} \subseteq \mathbb{R}^{\hat{X}}_b$ is compact.
\end{proposition}

The data set $\overline{\hat{\Phi}}$ endows $\hat{X}$ with a pseudo-metric structure where the pseudo-distance is given by \[ \overline{D}_{\hat{X}}(\hat{x}_1, \hat{x}_2) := \sup\limits_{\overline{\varphi} \in \overline{\hat{\Phi}}} \left \vert \overline{\varphi}(\hat{x}_1)-\overline{\varphi}(\hat{x}_2)\right \vert, \ \mathrm{for \ every} \ \hat{x}_1, \hat{x}_2 \in \hat{X}. \]

\begin{proposition}\label{DTild}
On $\hat{X}$, $\overline{D}_{\hat{X}} = D_{\hat{X}}$; so $\overline{D}_{\hat{X}}$ is a metric.
\end{proposition}
\begin{proof}
Let $\hat{x}_1, \hat{x}_2 \in \hat{X}$. 
By applying Proposition \ref{SpDns} to the continuous function $f(\overline{\varphi}):=\left\vert\overline{\varphi}(\hat x_1)-\overline{\varphi}(\hat x_2)\right\vert$, we get:
\begin{align*}
    \overline{D}_{\hat{X}}(\hat{x}_1, \hat{x}_2) & := \sup\limits_{\overline{\varphi} \in \overline{\hat{\Phi}}} \left \vert \overline{\varphi}(\hat{x}_1)-\overline{\varphi}(\hat{x}_2)\right \vert \\
    & = \sup\limits_{\hat{\varphi} \in \hat{\Phi}} \left \vert \hat{\varphi}(\hat{x}_1)-\hat{\varphi}(\hat{x}_2)\right \vert \\
    & =: D_{\hat{X}}(\hat{x}_1, \hat{x}_2).
\end{align*}
As $\hat{x}_1, \hat{x}_2 \in \hat{X}$ are arbitrary, we have the proposed equality.
\end{proof}

In view of Proposition \ref{DTild}, the symbol $\mathrm{Iso}(\hat{X})$ can be used without any ambiguity about the underlying metric.

By Theorem \ref{AutTopGp}, the set $\mathrm{Aut}_{\overline{\hat{\Phi}}}(\hat{X}) \subseteq \mathrm{Iso}(\hat{X})$ of all invertible $\overline{\hat{\Phi}}$-operations is a topological group with respect to the topology induced by the pseudo-distance $\overline{\hat{D}}_{\mathrm{Aut}} : \mathrm{Aut}_{\overline{\hat{\Phi}}}(\hat{X}) \times \mathrm{Aut}_{\overline{\hat{\Phi}}}(\hat{X}) \to \mathbb{R}$: 
\[ \overline{\hat{D}}_{\mathrm{Aut}}(\overline{g}, \overline{h}) := \sup_{\overline{\varphi} \in \overline{\hat{\Phi}}} D_{\overline{\hat{\Phi}}}(\overline{\varphi} \circ \overline{g}, \overline{\varphi} \circ \overline{h}), \ \mathrm{for \ every } \ \overline{g}, \overline{h} \in \mathrm{Aut}_{\overline{\hat{\Phi}}}(\hat{X}). \]
If $\overline{g}, \overline{h} \in \mathrm{Aut}_{\overline{\hat{\Phi}}}(\hat{X})$, then
\begin{align*}
    \overline{\hat{D}}_{\mathrm{Aut}}(\overline{g}, \overline{h}) & = \sup_{\overline{\varphi} \in \overline{\hat{\Phi}}} D_{\overline{\hat{\Phi}}}(\overline{\varphi} \circ \overline{g}, \overline{\varphi} \circ \overline{h}) \\
   & = \sup_{\overline{\varphi} \in \overline{\hat{\Phi}}} \sup_{\hat{x} \in \hat{X}} \left \vert \overline{\varphi}(\overline{g}(\hat{x})) - \overline{\varphi}(\overline{h}(\hat{x})) \right \vert \\
   & = \sup_{\hat{x} \in \hat{X}} \overline{D}_{\hat{X}}(\overline{g}(\hat{x}), \overline{h}(\hat{x})) \\
   & = \sup_{\hat{x} \in \hat{X}} D_{\hat{X}}(\overline{g}(\hat{x}), \overline{h}(\hat{x})).
\end{align*}
Therefore, if $\overline{\hat{D}}_{\mathrm{Aut}}(\overline{g}, \overline{h})=0$, then $D_{\hat{X}}(\overline{g}(\hat{x}), \overline{h}(\hat{x}))=0$ for every $\hat x\in \hat X$.
Since $D_{\hat{X}}$ is a metric, $\overline{g}(\hat{x}) = \overline{h}(\hat{x})$ for every $\hat{x} \in \hat{X}$, whence $\overline{g}=\overline{h}$. 
So, $(\mathrm{Aut}_{\overline{\hat{\Phi}}}(\hat{X}),\overline{\hat{D}}_{\mathrm{Aut}})$ is a metric space.

\begin{proposition}\label{ChgClCpPhOp}
Every $\check{g} \in \mathrm{Aut}_{\hat{\Phi}}(\hat{X})$ is a $\overline{\hat{\Phi}}-$operation. 
\end{proposition}

\begin{proof}
Let $\overline{\varphi} \in \overline{\hat{\Phi}}$ and $\check{g} \in \mathrm{Aut}_{\hat{\Phi}}(\hat{X})$. We show that $\overline{\varphi} \check{g} \in \overline{\hat{\Phi}}$.

There is a sequence $(\hat{\varphi}_n)_{n \in \mathbb{N}}$ in $\hat{\Phi}$ such that $\hat{\varphi}_n \rightarrow \overline{\varphi}$. As $\check{g}$ is a bijection of $\hat{X}$ by Proposition \ref{PhiOpIso}, we have \[ \| \hat{\varphi}_n \check{g} - \overline{\varphi} \check{g} \|_{\infty} = \| \hat{\varphi}_n - \overline{\varphi} \|_{\infty}; \] whence $\hat{\varphi}_n \check{g} \rightarrow \overline{\varphi} \check{g}$ in the space $C(\hat{X},\mathbb{R})$. As $\check{g}$ is a $\hat{\Phi}-$operation, $(\hat{\varphi}_n \check{g})_{n \in \mathbb{N}}$ is a sequence in the space $\hat{\Phi} \subseteq C(\hat{X},\mathbb{R})$. Consequently, $\overline{\varphi} \check{g} \in \overline{\hat{\Phi}}$.
\end{proof}

\begin{corollary}\label{CpSbstClCp}
$\hat{G} \subseteq \mathrm{Aut}_{\hat{\Phi}}(\hat{X}) \subseteq \mathrm{Aut}_{\overline{\hat{\Phi}}}(\hat{X})$.
\end{corollary}

\begin{proof}
The first inclusion is given by Corollary \ref{IndBjGp}. As for the second, let $\check{g} \in \mathrm{Aut}_{\hat{\Phi}}(\hat{X})$. By Proposition \ref{ChgClCpPhOp}, $\check{g}$ is a $\overline{\hat{\Phi}}-$operation. As $\mathrm{Aut}_{\hat{\Phi}}(\hat{X})$ is a group, $\check{g}^{-1} \in \mathrm{Aut}_{\hat{\Phi}}(\hat{X})$, and again by Proposition \ref{ChgClCpPhOp}, is a $\overline{\hat{\Phi}}-$operation. Consequently, $\check{g} \in \mathrm{Aut}_{\overline{\hat{\Phi}}}(\hat{X})$, and by the arbitrariness of $\check{g}$, we have the second inclusion.
\end{proof}

\begin{remark}\label{ClCpPhCpGPP}
Corollary \ref{CpSbstClCp} implicitly states that $(\overline{\hat{\Phi}}, \hat{G})$ is a perception pair.
\end{remark}

Let $\overline{\hat{G}}$ and $\overline{\mathrm{Aut}_{\hat{\Phi}}(\hat{X})}$ respectively denote the closures of $\hat{G}$ and $\mathrm{Aut}_{\hat{\Phi}}(\hat{X})$ in the space $\mathrm{Iso}(\hat{X})$ of all isometries of 
$(\hat{X},D_{\hat{X}})$. Recall that the topology on $\mathrm{Iso}(\hat{X})$ is given by the restriction to $\mathrm{Iso}(\hat{X})$ of the metric $\hat d_\infty$ defined on $C(\hat X,\hat X)$ by setting $\hat d_\infty(f,g) := \sup_{\hat x \in \hat X} D_{\hat X}(f(\hat x), g(\hat x)), \ \mathrm{for\ every\ } f, g \in C(\hat X,\hat X)$. Note also that the isometries of $\hat{X}$ with respect to the metric $D_{\hat{X}}$ coincide with those induced by the metric $\overline{D}_{\hat{X}}$ (Proposition \ref{DTild}).

\begin{corollary}\label{ClCpsbstClCp}
$\overline{\hat{G}} \subseteq \overline{\mathrm{Aut}_{\hat{\Phi}}(\hat{X})} \subseteq \mathrm{Aut}_{\overline{\hat{\Phi}}}(\hat{X}) \subseteq \mathrm{Iso}(\hat{X})$.
\end{corollary}
\begin{proof}
The last inclusion is given by Proposition \ref{PhiOpIso}. As we have seen at the beginning of Section \ref{CR}, $\hat{X}$ is compact, and $\overline{\hat{\Phi}}$ is compact by Proposition \ref{ClCpPhComp}. It follows from Theorem \ref{AutComp} that $\mathrm{Aut}_{\overline{\hat{\Phi}}}(\hat{X})$ is a closed subspace of $\mathrm{Iso}(\hat{X})$. Since $\hat{G} \subseteq \mathrm{Aut}_{\hat{\Phi}}(\hat{X}) \subseteq \mathrm{Aut}_{\overline{\hat{\Phi}}}(\hat{X})$ (Corollary \ref{CpSbstClCp}), we have the first two inclusions.

\end{proof}

\begin{remark}
Corollary \ref{ClCpsbstClCp} directly implies that the closure of $\hat{G}$ in $\mathrm{Aut}_{\overline{\hat{\Phi}}}(\hat{X})$ coincides with $\overline{\hat{G}}$, i.e., with the closure of $\hat{G}$ in the space $\mathrm{Iso}(\hat{X})$. Similarly the closure of $\mathrm{Aut}_{\hat{\Phi}}(\hat{X})$ in $\mathrm{Aut}_{\overline{\hat{\Phi}}}(\hat{X})$ coincides with $\overline{\mathrm{Aut}_{\hat{\Phi}}(\hat{X})}$.
\end{remark}

Incidentally, Proposition \ref{IsoXComp} and Corollary \ref{ClCpsbstClCp} also give that the spaces $\overline{\hat{G}}$ and $\overline{\mathrm{Aut}_{\hat{\Phi}}(\hat{X})}$ are compact.

\begin{proposition}\label{ClCpGTpGp}
The groups $\hat{G}$ and $\overline{\hat{G}}$ are both topological subgroups of the compact group $\mathrm{Aut}_{\overline{\hat{\Phi}}}(\hat{X})$.
\end{proposition}
\begin{proof}
By Corollary \ref{ClCpsbstClCp}, we have $\hat{G} \subseteq \overline{\hat{G}} \subseteq \mathrm{Aut}_{\overline{\hat{\Phi}}}(\hat{X})$. By Theorem \ref{AutTopGp}, $\mathrm{Aut}_{\overline{\hat{\Phi}}}(\hat{X})$ is a topological group. So, $\hat{G}$ and $\overline{\hat{G}}$, being subgroups of a topological group are likewise topological. The compactness of $\mathrm{Aut}_{\overline{\hat{\Phi}}}(\hat{X})$ is given by Theorem \ref{AutComp}.
\end{proof}

We can now state

\begin{theorem}\label{PPCp}
Given any perception pair $(\Phi,G)$, $\mathrm{dom}(\Phi) = X$ with totally bounded $\Phi$ endowing $X$ with a metric structure, the perception pair $(\overline{\hat{\Phi}}, \overline{\hat{G}})$, $\mathrm{dom}(\overline{\hat{\Phi}}) = \hat{X}$ is compact.
\end{theorem}
\begin{proof}
Propositions \ref{ClCpPhComp} and \ref{ClCpGTpGp} together give the assertion.
\end{proof}

\begin{definition}
We say that the perception pair $(\Phi, G)$ with $\mathrm{dom}(\Phi) = X$ is \textit{isometrically embedded} into the perception pair $(\Phi^*, G^*)$ with $\mathrm{dom}(\Phi^*) = X^*$ if there are isometric embeddings $j^* : X \to X^*$, $i^* : \Phi \to \Phi^*$, and $k^* : G \to G^*$ such that the images $j^*(X)$, $i^*(\Phi)$, and $k^*(G)$ are all dense in $X^*$, $\Phi^*$, and $G^*$ respectively, and the following commutativity conditions are satisfied:
$i^*(\varphi) \circ j^* = \varphi$ for every $\varphi \in \Phi$ and $k^*(g) \circ j^* = j^* \circ g$ for every $g \in G$. If $(\Phi^*, G^*)$ is compact, it is said to be a \textit{compactification} of $(\Phi, G)$.
\end{definition}

With this definition at our disposal, we summarize

\begin{theorem}\label{CompPercPa}
Every perception pair $(\Phi, G)$, $\mathrm{dom}(\Phi) = X$, with totally bounded $\Phi$ endowing $X$ with a metric structure, admits a compactification $(\Phi^*, G^*)$, $\mathrm{dom}(\Phi^*) = X^*$.
\end{theorem}
\begin{proof}
Put $\Phi^* := \overline{\hat{\Phi}}$ and $G^* := \overline{\hat{G}}$ in Theorem \ref{PPCp}.
\end{proof}


\subsection{GENEOs and Completions}\label{GNsCompls}

Our next goal is to construct compactifications of the spaces $\mathcal{F} \subseteq \mathcal{F}^{\mathrm{all}}_{T}$ of GENEOs $(F,T) : (\Phi,G) \to (\Psi,H)$ with the property that the images $F(\Phi), F \in \mathcal{F}$ form a cover for the data set $\Psi$, while maintaining the assumptions that $\Phi$ and $\Psi$ are totally bounded and endow $X$ and $Y$ with metric structures. We have shown, under these assumptions, that the perception pairs $(\Phi,G)$, $\mathrm{dom}(\Phi) = X$ and $(\Psi,H)$, $\mathrm{dom}(\Psi) = Y$ can be embedded nicely into the perception pairs $(\hat{\Phi},\hat{G})$, $\mathrm{dom}(\hat{\Phi}) = \hat{X}$ and $(\hat{\Psi},\hat{H})$, $\mathrm{dom}(\hat{\Psi}) = \hat{Y}$, respectively, through the compatible isometries $(j_1, j_2) : (X, Y) \to (\hat{X}, \hat{Y})$, $(i_1, i_2) : (\Phi, \Psi) \to (\hat{\Phi}, \hat{\Psi})$, and $(k_1, k_2) : (G, H) \to (\hat{G}, \hat{H})$, and are now in a position to use the GENEOs $(F,T) : (\Phi,G) \to (\Psi,H)$ to define new GENEOs $(\hat{F},\hat{T}) : (\hat{\Phi},\hat{G}) \to (\hat{\Psi},\hat{H})$. Our construction will be further extended to the GENEOs $(\overline{\hat{F}},\hat{T}) : (\overline{\hat{\Phi}}, \hat{G}) \to (\overline{\hat{\Psi}},\hat{H})$ and $(\overline{\hat{F}},\overline{\hat{T}}) : (\overline{\hat{\Phi}},\overline{\hat{G}}) \to (\overline{\hat{\Psi}},\overline{\hat{H}})$ later.


\subsubsection{The Induced GENEOs}
\label{CpF}
Let $(F,T) : (\Phi,G) \to (\Psi,H)$ be a GENEO in $\mathcal{F} \subseteq \mathcal{F}^{\mathrm{all}}_{T}$. We put \[ \hat{F}(\hat{\varphi}) := \widehat{F(\varphi)}, \] and \[ \hat{T}(\hat{g}) := \widehat{T(g)}, \] where $\varphi \in \Phi$, $g \in G$, and $F \in \mathcal{F} \subseteq \mathcal{F}_{T}^{all}$.

The maps $\hat{F} : \hat{\Phi} \to \hat{\Psi}$ and $\hat{T} : \hat{G} \to \hat{H}$ are clearly well defined, since the maps $i_1:\Phi\to\hat \Phi$ (taking $\varphi$ to $\hat\varphi$) and $k_1:G\to\hat G$ (taking $g$ to $\hat g$) are injective. Moreover,

\begin{remark}
The map $\hat{F}$ is injective if and only if $F \in \mathcal{F}$ is an injection. As $i_2 : \Psi \to \hat{\Psi}$ is injective, we have \[ \hat{F}(\hat{\varphi}_1) = \hat{F}(\hat{\varphi}_2) \xLeftRightarrow[\mathrm{def}]{} \widehat{F(\varphi_1)} = \widehat{F(\varphi_2)} \iff F(\varphi_1) = F(\varphi_2), \] for every $\varphi_1, \varphi_2 \in \Phi$. The injectivity of $i_1 : \Phi \to \hat{\Phi}$, together with these equivalences, gives the assertion.
\end{remark}

Recalling Propositions \ref{IndBjCmt} and \ref{IndBjFnCompCmt}, we prove

\begin{proposition}\label{CpFGO}
The map $\hat{F} : \hat{\Phi} \to \hat{\Psi}$ is a GENEO with respect to $\hat{T} : \hat{G} \to \hat{H}$.
\end{proposition}

\begin{proof}
It is easy to see that $\hat{T} : \hat{G} \to \hat{H}$ is a group homomorphism: If $a, b \in G$, then

\begin{align*}
     \hat{T}(\hat{a}\hat{b}) & = \hat{T}(\widehat{ab}) \\
     & = \widehat{T(ab)} \\
     & = \widehat{T(a)T(b)} \\
     & = \widehat{T(a)}\widehat{T(b)} \\
     & = \hat{T}(\hat{a})\hat{T}(\hat{b}).
\end{align*}

Similarly, if $\varphi \in \Phi$, $g \in G$, and $F \in \mathcal{F}$, we have:

\begin{align*}
    \hat{F}(\hat{\varphi} \hat{g}) & = \hat{F}(\widehat{\varphi g}) \\
    & = \widehat{F(\varphi g)} \\
    & = \widehat{F(\varphi)T(g)} \\
    & = \widehat{F(\varphi)}\widehat{T(g)} \\
    & = \hat{F}(\hat{\varphi})\hat{T}(\hat{g}).
\end{align*}
So, $\hat{F}$ is $\hat{T}-$equivariant.

Now, let $\varphi_1, \varphi_2 \in \Phi$. As $i_1 : \Phi \to \hat{\Phi}$ and $i_2 : \Psi \to \hat{\Psi}$ are isometries (Theorem \ref{ExtSgIso}) and $F : \Phi \to \Psi$ is non-expansive,
\begin{align*}
D_{\hat{\Psi}}(\hat{F}(\hat{\varphi}_1),\hat{F}(\hat{\varphi}_2)) & = D_{\hat{\Psi}}(\widehat{F(\varphi_1)},\widehat{F(\varphi_2)}) \\ & = D_{\Psi}(F(\varphi_1),F(\varphi_2)) \\ & \leq D_{\Phi}(\varphi_1,\varphi_2) \\ & = D_{\hat{\Phi}}(\hat{\varphi}_1, \hat{\varphi}_2)
\end{align*}
 whence $\hat{F}$ is non-expansive.

\end{proof}

Let us put \[ \mathcal{F}_1 := \{ \hat{F}: \hat{\Phi} \to \hat{\Psi} \mid F \in \mathcal{F} \}, \] and define a map $f_1 : \mathcal{F} \to \mathcal{F}_1$ by setting \[ f_1(F) := \hat{F}. \]

The set $\mathcal{F}^{\mathrm{all},1}_{\hat{T}} \supseteq \mathcal{F}_1$ of all GENEOs from $(\hat{\Phi},\hat{G})$ to $(\hat{\Psi},\hat{H})$ with respect to the homomorphism $\hat{T} : \hat{G} \to \hat{H}$ is a metric space with the distance function $D^{1}_{\mathrm{GENEO}}$ given by \[ D^{1}_{\mathrm{GENEO}}(F^{\prime}, F^{\prime \prime}) := \sup_{\hat{\varphi} \in \hat{\Phi}} D_{\hat{\Psi}}(F^{\prime}(\hat{\varphi}), F^{\prime \prime}(\hat{\varphi})), \ \mathrm{ for\ every\ } F^{\prime}, F^{\prime \prime} \in \mathcal{F}^{\mathrm{all},1}_{\hat{T}}. \]

\begin{proposition}\label{f1Iso}
The correspondence $f_1 : \mathcal{F} \to \mathcal{F}_1$ is an isometry with respect to the distances $D_{\mathrm{GENEO}}$ and $D^{1}_{\mathrm{GENEO}}$.
\end{proposition}

\begin{proof}
The map $f_1$ is surjective by construction. Let $F_1, F_2 \in \mathcal{F}$ be distinct GENEOs; i.e., there is a $\varphi \in \Phi$ such that $F_1(\varphi) \neq F_2(\varphi)$. As $i_2 : \Psi \to \hat{\Psi}$ is injective, $\hat{F}_1(\hat{\varphi}) := \widehat{F_1(\varphi)} \neq  \widehat{F_2(\varphi)} =: \hat{F}_2(\hat{\varphi})$, whence $f_1(F_1) := \hat{F}_1 \neq \hat{F}_2 =: f_1(F_2)$, and $f_1 : \mathcal{F} \to \mathcal{F}_1$ is injective.

We now show that $f_1$ preserves distances. If $F_1, F_2 \in \mathcal{F}$, by applying Proposition \ref{SpDns} to the real-valued continuous function $f(\hat{\varphi}):=D_{\hat{\Psi}}(\hat{F}_1(\hat{\varphi}),\hat{F}_2(\hat{\varphi}))$,  we get
\begin{align*}
    D^{1}_{\mathrm{GENEO}}(\hat{F}_1,\hat{F}_2) & := \sup_{\hat{\varphi} \in \hat{\Phi}} D_{\hat{\Psi}}(\hat{F}_1(\hat{\varphi}),\hat{F}_2(\hat{\varphi})) \\
    & = \sup_{\varphi \in \Phi} D_{\hat{\Psi}}(\hat{F}_1(\hat{\varphi}),\hat{F}_2(\hat{\varphi}))  \\
    & = \sup_{\varphi \in \Phi} D_{\hat{\Psi}}(\widehat{F_1(\varphi)},\widehat{F_2(\varphi)}) \\
    & = \sup_{\varphi \in \Phi} D_{\Psi}(F_1(\varphi),F_2(\varphi)) \\
    & =: D_{\mathrm{GENEO}}(F_1,F_2),
\end{align*}
as $i_2 : \Psi \to \hat{\Psi}$ is an isometry by Theorem \ref{ExtSgIso}. So, the bijection $f_1$ is an isometry.
\end{proof}

From the definitions of $\hat{F}$ and $\hat{T}$, it is already clear that the following natural commutativity conditions are trivially satisfied: 

\begin{proposition}\label{ComGN1}
For each $F \in \mathcal{F}$, 
\[ i_2 \circ F = f_1(F) \circ i_1, \mathrm{\ (i.e.,\ } \widehat{F(\varphi)}=\hat F(\hat\varphi)\mathrm{\ for\ every\ } \varphi\in \Phi\mathrm{)}\] 
and 
\[ k_2 \circ T = \hat{T} \circ k_1 \mathrm{\ (i.e.,\ } \widehat{T(g)}=\hat T(\hat g)\mathrm{\ for\ every\ } g\in G\mathrm{)}. \]
\end{proposition}

\subsection{Compactification of the Spaces of GENEOs}
\label{CompF}
We can now extend our construction from $(\hat{F}, \hat{T}) : (\hat{\Phi}, \hat{G}) \to (\hat{\Psi}, \hat{H})$ to $(\overline{\hat{F}},\hat{T}) : (\overline{\hat{\Phi}}, \hat{G}) \to (\overline{\hat{\Psi}},\hat{H})$ and $(\overline{\hat{F}},\overline{\hat{T}}) : (\overline{\hat{\Phi}},\overline{\hat{G}}) \to (\overline{\hat{\Psi}},\overline{\hat{H}})$ successively, while maintaining the assumptions of Section \ref{GNsCompls}. First, we show that $\hat{F} : \hat{\Phi} \to \hat{\Psi}$ induces a non-expansive $\hat{T}$-equivariant map $\overline{\hat{F}} : \overline{\hat{\Phi}} \to \overline{\hat{\Psi}}$; then we will use the assumption that the family of sets $\{F(\Phi) \mid F \in \mathcal{F}\}$ covers $\Psi$ to define a group homomorphism $\overline{\hat{T}} : \overline{\hat{G}} \to \overline{\hat{H}}$ with respect to which  $\overline{\hat{F}}$ remains equivariant.

Let us define a map $\overline{\hat{F}} : \overline{\hat{\Phi}} \to \overline{\hat{\Psi}}$ as follows. Let $\overline{\varphi} \in \overline{\hat{\Phi}}$; then there is a sequence $(\hat{\varphi}_n)_{n \in \mathbb{N}}$ in $\hat{\Phi}$ such that $\hat{\varphi}_n \rightarrow \overline{\varphi}$ with respect to the uniform norm. As $\hat{F}$ is non-expansive, $(\hat{F}(\hat{\varphi}_n))_{n \in \mathbb{N}}$ is a Cauchy sequence in $\hat{\Psi} \subseteq \overline{\hat{\Psi}}$; so it converges to some $\overline{\psi}$ in the complete space $\overline{\hat{\Psi}}$. Let us put \[ \overline{\hat{F}}(\overline{\varphi}) := \overline{\psi}. \] That is, \[ \overline{\hat{F}}(\lim_{n \rightarrow \infty} \hat{\varphi}_n) := \lim_{n \rightarrow \infty} \hat{F}(\hat{\varphi}_n). \]

Note that since $\hat{F}$ is non-expansive, the map $\overline{\hat{\mathcal{F}}}$ does not depend on the sequence $(\hat{\varphi}_n)_{n \in \mathbb{N}}$ converging to $\overline{\varphi}$, and is therefore well defined. Moreover, $\overline{\hat{F}} \vert_{\hat{\Phi}} = \hat{F}$.

\begin{proposition}\label{ClCpFNExp}
The map $\overline{\hat{F}} : \overline{\hat{\Phi}} \to \overline{\hat{\Psi}}$ is a GENEO with respect to $\hat{T} : \hat{G} \to \hat{H}$.
\end{proposition}
\begin{proof}
Let $\overline{\varphi}_1, \overline{\varphi}_2 \in \overline{\hat{\Phi}}$; then there are sequences $(\hat{\varphi}_{1,n})_{n \in \mathbb{N}}$ and $(\hat{\varphi}_{2,n})_{n \in \mathbb{N}}$ in $\hat{\Phi}$ such that $\hat{\varphi}_{1,n} \rightarrow \overline{\varphi}_1$ and $\hat{\varphi}_{2,n} \rightarrow \overline{\varphi}_2$. Recalling that $\hat{F}$ is non-expansive, we compute:
\begin{align*}
    \left\| \overline{\hat{F}}(\overline{\varphi}_1) - \overline{\hat{F}}(\overline{\varphi}_2) \right\|_\infty
    &= \left\| \lim_{n \rightarrow \infty} \hat{F}(\hat{\varphi}_{1,n}) - \lim_{n \rightarrow \infty} \hat{F}(\hat{\varphi}_{2,n}) \right\|_\infty \\
    &= \lim_{n \rightarrow \infty} \left\|\hat{F}(\hat{\varphi}_{1,n}) - \hat{F}(\hat{\varphi}_{2,n})\right\|_\infty \\
    & \leq \lim_{n \rightarrow \infty} \left\| \hat{\varphi}_{1,n} - \hat{\varphi}_{2,n} \right\|_\infty \\
    &= \left\| \lim_{n \rightarrow \infty} \hat{\varphi}_{1,n} - \lim_{n \rightarrow \infty} \hat{\varphi}_{2,n} \right\|_\infty \\
    &= \| \overline{\varphi}_1 - \overline{\varphi}_2 \|_\infty;
\end{align*}
so, $\overline{\hat{F}}$ is non-expansive.

Let $\overline{\varphi} \in \overline{\hat{\Phi}}$ and $\hat{g} \in \hat{G}$. Then there is a sequence $(\hat{\varphi}_n)_{n \in \mathbb{N}}$ in $\hat{\Phi}$ such that $\hat{\varphi}_n \rightarrow \overline{\varphi}$ with respect to the uniform norm; consequently, $\hat{\varphi}_n \hat{g} \rightarrow \overline{\varphi} \hat{g}$. As $\hat{F}$ is $\hat{T}-$equivariant (Proposition \ref{CpFGO}) and the action of $\hat{H}$ on $\overline{\hat{\Psi}}$ is continuous (Theorem \ref{AutTopGp}), we have
\begin{align*}
    \overline{\hat{F}}(\overline{\varphi} \circ \hat{g})
    &= \overline{\hat{F}}(\lim_{n \rightarrow \infty} \hat{\varphi}_n \circ \hat{g})\\
    &= \lim_{n \rightarrow \infty} \hat{F}(\hat{\varphi}_n \circ \hat{g})\\
    &= \lim_{n \rightarrow \infty} (\hat{F}(\hat{\varphi}_n) \circ \hat{T}(\hat{g}))\\
    &= (\lim_{n \rightarrow \infty} \hat{F}(\hat{\varphi}_n)) \circ \hat{T}(\hat{g})\\
    &= \overline{\hat{F}}(\overline{\varphi}) \circ \hat{T}(\hat{g}),
\end{align*}
whence $\overline{\hat{F}}$ is $\hat{T}-$equivariant and the proposition is proved.
\end{proof}

Let us put \[ \mathcal{F}_2 := \left\{ \overline{\hat{F}}: \overline{\hat{\Phi}} \to \overline{\hat{\Psi}} \mid F \in \mathcal{F} \right\}, \] and define a map $f_2 : \mathcal{F}_1 \to \mathcal{F}_2$ by setting \[ f_2(\hat{F}) :=  \overline{\hat{F}}. \]



The set $\mathcal{F}^{\mathrm{all}, 2}_{\hat{T}} \supseteq \mathcal{F}_2$ of all GENEOs from $(\overline{\hat{\Phi}}, \hat{G})$ to $(\overline{\hat{\Psi}},\hat{H})$ with respect to the homomorphism $\hat{T} : \hat{G} \to \hat{H}$ is a metric space with the distance function $D^{2}_{\mathrm{GENEO}}$ given by 

\[ D^{2}_{\mathrm{GENEO}}\left(F^{\prime}, F^{\prime \prime} \right) := 
\sup_{\overline{\varphi} \in \overline{\hat{\Phi}}} D_{\overline{\hat{\Psi}}} \left(F^{\prime}(\overline{\varphi}), F^{\prime \prime}(\overline{\varphi}) \right), \ \mathrm{ for\ every\ } F^{\prime}, F^{\prime \prime} \in \mathcal{F}^{\mathrm{all},2}_{\hat{T}}. \]

\begin{proposition}\label{f2Iso}
The correspondence $f_2 : \mathcal{F}_1 \to \mathcal{F}_2$ is an isometry with respect to the distances $D^{1}_{\mathrm{GENEO}}$ and $D^{2}_{\mathrm{GENEO}}$.
\end{proposition}
\begin{proof}
The map $f_2 : \mathcal{F}_1 \to \mathcal{F}_2$ is surjective by construction. Also, if $\hat{F}_1, \hat{F}_2 \in \mathcal{F}_1$ are distinct, i.e., there is a $\hat{\varphi} \in \hat{\Phi}$ with $\hat{F}_1(\hat{\varphi}) \neq \hat{F}_2(\hat{\varphi})$, then $f_2(\hat{F}_1)(\hat{\varphi}) = \overline{\hat{F}}_1(\hat{\varphi}) \neq \overline{\hat{F}}_2(\hat{\varphi}) = f_2(\hat{F}_2)(\hat{\varphi})$ since we respectively have $\overline{\hat{F}}_1 \vert_{\hat{\Phi}} = \hat{F}_1$ and $\overline{\hat{F}}_2 \vert_{\hat{\Phi}} = \hat{F}_2$; whence $f_2(\hat{F}_1)(\hat{\varphi}) \neq f_2(\hat{F}_2)(\hat{\varphi})$ and $f_2$ is injective as well.

If $\overline{\hat{F}}_1,\overline{\hat{F}}_2 \in \mathcal{F}_2 \subseteq \mathcal{F}^{\mathrm{all}, 2}_{\hat{T}}$, by applying Proposition \ref{SpDns} to the real-valued continuous function $f(\overline{\varphi}):=D_{\overline{\hat{\Psi}}}\left(\overline{\hat{F}}_1\left(\overline{\varphi}\right),\overline{\hat{F}}_2\left(\overline{\varphi}\right)\right)$,  we get
\begin{align*}
    D^{2}_{\mathrm{GENEO}}\left(\overline{\hat{F}}_1,\overline{\hat{F}}_2\right) & := \sup_{\overline{\varphi} \in \overline{\hat{\Phi}}} D_{\overline{\hat{\Psi}}}\left(\overline{\hat{F}}_1\left(\overline{\varphi}\right),\overline{\hat{F}}_2\left(\overline{\varphi}\right)\right) \\
    & = \sup_{\hat{\varphi} \in \hat{\Phi}} D_{\overline{\hat{\Psi}}}\left(\overline{\hat{F}}_1\left(\hat{\varphi}\right),\overline{\hat{F}}_2\left(\hat{\varphi}\right)\right) \\
    & = \sup_{\hat{\varphi} \in \hat{\Phi}} D_{\hat{\Psi}}\left(\hat{F}_1\left(\hat{\varphi}\right),\hat{F}_2\left(\hat{\varphi}\right)\right) \\
    & =: D^{1}_{\mathrm{GENEO}}\left(\hat{F}_1,\hat{F}_2\right).
\end{align*}
as $D_{\hat{\Psi}}$ and $D_{\overline{\hat{\Psi}}}$ both are restrictions of the distance induced by the uniform norm on $\mathbb{R}^{Y}_b$ to $\hat{\Psi}$ and $\overline{\hat{\Psi}}$ respectively. So, the bijection $f_2$ is an isometry.
\end{proof}

As $\overline{\hat{F}} \vert_{\hat{\Phi}} = \hat{F}$ for each $F \in \mathcal{F}$, by Proposition \ref{ComGN1} we have

\begin{proposition}\label{ComGN2}
For each $F \in \mathcal{F}$, \[ i_2 \circ F = (f_2 \circ f_1(F)) \circ i_1. \] That is, for every $\varphi\in\Phi$, \[ \widehat{F(\varphi)} = \overline{\hat{F}}(\hat{\varphi}). \]
\end{proposition}

\vspace{0.3cm}

Let us now utilize the assumption that $\{ F(\Phi) \mid F \in \mathcal{F} \}$ covers $\Psi$ to define a homomorphism $\overline{\hat{T}} : \overline{\hat{G}} \to \overline{\hat{H}}$. First we need

\begin{definition}
We say that a space $\mathcal{F} \subseteq \mathcal{F}_T^{\mathrm{all}}$ of GENEOs $(F, T) : (\Phi, G) \to (\Psi, H)$ is \textit{collectionwise surjective} if for each $\psi \in \Psi$, there exist an $F_{\psi} \in \mathcal{F}$ and a $\varphi_\psi\in\Phi$ such that $F_{\psi}(\varphi_{\psi}) = \psi$; that is, $\bigcup_{F \in \mathcal{F}} F(\Phi) = \Psi$.
\end{definition}

The key property of collectionwise surjective spaces of GENEOs is given in

\begin{theorem}\label{TnExp}
If the space $\mathcal{F} \subseteq \mathcal{F}_T^{\mathrm{all}}$ of GENEOs $(F, T) : (\Phi, G) \to (\Psi, H)$ is collectionwise surjective, then the homomorphism $T$ is non-expansive.
\end{theorem}
\begin{proof}
Let $a, b \in G$; then as $\mathcal{F}$ is collectionwise surjective and each $F \in \mathcal{F}$ is a GENEO, we have
\begin{align*}
    D_{\mathrm{Aut}}(T(a), T(b)) & := \sup_{\psi \in \Psi} D_{\Psi}(\psi T(a), \psi T(b)) \\
    & = \sup_{\psi \in \Psi} D_{\Psi}(F_{\psi}(\varphi_{\psi}) T(a), F_{\psi}(\varphi_{\psi}) T(b)) \\
    & = \sup_{\psi \in \Psi} D_{\Psi}( F_{\psi}(\varphi_{\psi} a), F_{\psi}(\varphi_{\psi} b) ) \\
    & \leq \sup_{\psi \in \Psi} D_{\Phi}( \varphi_{\psi} a, \varphi_{\psi} b ) \\
    & \leq \sup_{\varphi \in \Phi} D_{\Phi}( \varphi a, \varphi b ) \\
    & = D_{\mathrm{Aut}}(a, b).
\end{align*}
\end{proof}

For the rest of this section, the spaces $\mathcal{F} \subseteq \mathcal{F}_T^{\mathrm{all}}$ will be assumed to be collectionwise surjective.
Clearly, $\mathcal{F}_1 := \{ \hat{F} : \hat{\Phi} \to \hat{\Psi} \mid F \in \mathcal{F} \}$ is collectionwise surjective whenever $\mathcal{F}$ is so, since $\hat F_\psi(\hat\varphi_\psi)=\widehat{F_\psi(\varphi_\psi)}=\hat\psi$ for every $\psi\in\Psi$. 

Theorem \ref{TnExp} implies

\begin{corollary}\label{CpTNExp}
The homomorphism $\hat{T} : \hat{G} \to \hat{H}$ is non-expansive.
\end{corollary}

Corollary \ref{CpTNExp} allows us to define a map $\overline{\hat{T}} : \overline{\hat{G}} \to \overline{\hat{H}}$ unambiguously: Let $\overline{g} \in \overline{\hat{G}}$ and $(\hat{g}_n)_{n \in \mathbb{N}}$ be a sequence in $\hat{G}$ that converges to $\overline{g}$ in $\overline{\hat{G}}$. 
As $\hat{T}$ is non-expansive and $\overline{\hat{H}}$ is a complete metric space, the sequence  $(\hat{T}(\hat{g}_n))_{n \in \mathbb{N}}$ in $\hat H$  converges to a unique element $\overline{h} \in \overline{\hat{H}}$. We put \[ \overline{\hat{T}}(\overline{g}) := \overline{h}. \] 
That is, \[ \overline{\hat{T}}\left(\lim_{n \to \infty} \hat{g}_n\right) := \lim_{n \to \infty} \hat{T}(\hat{g}_n). \]

Note that $\overline{\hat{T}} \vert_{\hat{G}} = \hat{T}$. So, the commutativity condition $k_2 \circ T = \hat{T} \circ k_1$ (i.e.,\  $\widehat{T(g)}=\hat T(\hat g)\mathrm{\ for\ every\ } g\in G$) in Proposition \ref{ComGN1} can be rephrased as

\begin{proposition}\label{ComGNT}

$k_2 \circ T = \overline{\hat{T}} \circ k_1  \mathrm{\ (i.e.,\ } \widehat{T(g)}=\overline{\hat{T}}(\hat g)\mathrm{\ for\ every\ } g\in G\mathrm{). }
$
\end{proposition}

We observe that the map $\overline{\hat{T}} : \overline{\hat{G}} \to \overline{\hat{H}}$
preserves the group structure:
\begin{theorem}\label{thmgrouphom}
The function $\overline{\hat{T}} : \overline{\hat{G}} \to \overline{\hat{H}}$ is 
a group homomorphism.
\end{theorem}
\begin{proof}
Let $\overline{a}, \overline{b} \in \overline{\hat{G}}$, and $(\hat{a}_n)_{n \in \mathbb{N}}, (\hat{b}_n)_{n \in \mathbb{N}}$ be sequences in $\hat{G}$ converging respectively to $\overline{a}, \overline{b}$ in $\overline{\hat{G}}$.
%
%
%
Recalling the continuity of the composition of functions on $\mathrm{Aut}_{\hat{\Phi}}(\hat{X})$ and $\mathrm{Aut}_{\hat{\Psi}}(\hat{Y})$ (Theorem \ref{AutTopGp}) and the definition of  $\overline{\hat{T}}$, we compute
\begin{align*}
    \overline{\hat{T}}\left(\overline{a} \overline{b}\right) & = \overline{\hat{T}}\left(\lim_{n \to \infty} \hat{a}_n \lim_{n \to \infty} \hat{b}_n\right) \\
    & = \overline{\hat{T}}\left(\lim_{n \to \infty} \hat{a}_n   \hat{b}_n\right) \\
    & = \lim_{n \to \infty} \hat{T}(\hat{a}_n \hat{b}_n) \\
    & = \lim_{n \to \infty} \hat{T}(\hat{a}_n) \hat{T}(\hat{b}_n) \\
    & = \lim_{n \to \infty} \hat{T}(\hat{a}_n) \lim_{n \to \infty} \hat{T}(\hat{b}_n) \\
    & = \overline{\hat{T}}\left(\lim_{n \to \infty} \hat{a}_n\right) \overline{\hat{T}}\left(\lim_{n \to \infty} \hat{b}_n\right) \\
    & = \overline{\hat{T}}(\overline{a}) \overline{\hat{T}}(\overline{b}).
\end{align*}
Therefore, 
$\overline{\hat{T}}$ is a group homomorphism.
\end{proof}

The next claim allows us to pass from $\hat{T}-$equivariance to $\overline{\hat{T}}-$equivariance.

\begin{theorem}\label{thmextequiv}
Every GENEO $\overline{\hat{F}} \in \mathcal{F}_2 \subseteq \mathcal{F}^{\mathrm{all}, 2}_{\hat{T}}$ is $\overline{\hat{T}}-$equivariant as well. 
Hence $\left(\overline{\hat{F}}, \overline{\hat{T}}\right) : \left(\overline{\hat{\Phi}}, \overline{\hat{G}}\right) \to \left(\overline{\hat{\Psi}}, \overline{\hat{H}}\right)$ is a GENEO for each $F \in \mathcal{F}$.
\end{theorem}
\begin{proof}
Let $\overline{\varphi} \in \overline{\hat{\Phi}}, \overline{g} \in \overline{\hat{G}}$, and $(\hat{g}_n)_{n \in \mathbb{N}}$ be a sequence in $\hat{G}$ converging to $\overline{g}$.
Recalling the fact that $\overline{\hat{F}}$ is a GENEO for $\hat T$  (and in particular a non-expansive, and hence continuous, map) by Proposition \ref{ClCpFNExp}, the continuity of the actions of $\mathrm{Aut}_{\overline{\hat{\Phi}}}(\hat{X})$ and $\mathrm{Aut}_{\overline{\hat{\Psi}}}(\hat{Y})$ respectively on $\overline{\hat{\Phi}}$ and $\overline{\hat{\Psi}}$ (Theorem \ref{AutTopGp}), and the definition of  $\overline{\hat{T}}$, we compute
\begin{align*}
    \overline{\hat{F}}\left( \bar{\varphi}  \bar{g}\right) & = \overline{\hat{F}}\left(\overline{\varphi} \lim_{n \to \infty} \hat{g}_n \right) \\
    & = \overline{\hat{F}}\left(\lim_{n \to \infty} \overline{\varphi} \hat{g}_n \right) \\
    & = \lim_{n \to \infty} \overline{\hat{F}}\left(\overline{\varphi} \hat{g}_n\right) \\
    & = \lim_{n \to \infty} \overline{\hat{F}}(\overline{\varphi}) \hat{T}(\hat{g}_n) \\
    & = \overline{\hat{F}}(\overline{\varphi}) \lim_{n \to \infty} \hat{T}(\hat{g}_n) \\
    & = \overline{\hat{F}}(\overline{\varphi}) \overline{\hat{T}}\left(\lim_{n \to \infty} \hat{g}_n\right) \\
    & = \overline{\hat{F}}(\overline{\varphi}) \overline{\hat{T}}(\overline{g}).
\end{align*}
\end{proof}

Because of Theorems \ref{thmgrouphom} and \ref{thmextequiv}, we can now consider 
$\mathcal{F}_2$ as a set of GENEOs
from $\left(\overline{\hat{\Phi}},\overline{\hat{G}}\right)$ to $\left(\overline{\hat{\Psi}},\overline{\hat{H}}\right)$ with respect to $\overline{\hat T}$, and denote this set by $\mathcal{F}_3$ to make clear that we are taking  the homomorphism 
$\overline{\hat T}$ instead of ${\hat T}$.



The set $\mathcal{F}^{\mathrm{all}}_{\overline{\hat{T}}} \supseteq \mathcal{F}_3$ of all GENEOs from  $(\overline{\hat{\Phi}},\overline{\hat{G}})$ to $(\overline{\hat{\Psi}},\overline{\hat{H}})$ with respect to $\overline{\hat T}$ is a metric space with the distance function  $D^{3}_{\mathrm{GENEO}}$ given by \[ D^{3}_{\mathrm{GENEO}}(F', F'') := \sup_{\overline{\varphi} \in \overline{\hat{\Phi}}} D_{\overline{\hat{\Psi}}}(F'(\overline{\varphi}), F''(\overline{\varphi})), \ F', F'' \in \mathcal{F}^{\mathrm{all}}_{\overline{\hat{T}}}. \] Moreover, since the data sets $\overline{\hat{\Phi}}$ and $\overline{\hat{\Psi}}$ are compact, the space $(\mathcal{F}^{\mathrm{all}}_{\overline{\hat{T}}}, D^{3}_{\mathrm{GENEO}})$ is compact as well \cite[Theorem 7]{MaIntPa}. Consequently,

\begin{proposition}\label{TdHtFComp}
The closure $\mathrm{cl}(\mathcal{F}_3)$ of $\mathcal{F}_3 \subseteq \mathcal{F}^{\mathrm{all}}_{\overline{\hat{T}}}$ in the compact space $\mathcal{F}_{\overline{\hat{T}}}^{\mathrm{all}}$ is compact.
\end{proposition}

As the definitions of $D^{2}_{\mathrm{GENEO}}$ and $D^{3}_{\mathrm{GENEO}}$ do not depend on the reference homomorphisms $\hat{T}$ and $\overline{\hat{T}}$ respectively, we observe that the identity from $\mathcal{F}_2$ to $\mathcal{F}_3$ is an isometry.


Propositions \ref{f1Iso} and \ref{f2Iso} together give

\begin{proposition}\label{ComGN4}
The correspondence $f : \mathcal{F} \to \mathcal{F}_3$ given by \[ f := f_2 \circ f_1 \] is an isometry.
\end{proposition}

Therefore,  we can  rephrase Proposition \ref{ComGN2} as

\begin{proposition}\label{ComGNF}
For each $F \in \mathcal{F}$, \[ i_2 \circ F = f(F) \circ i_1
\mathrm{\ (i.e.,\ } \widehat{F(\varphi)}=\overline{\hat F}(\hat\varphi)\mathrm{\ for\ every\ } \varphi\in \Phi\mathrm{).} \]
\end{proposition}
 
We can now state the main result in this paper by introducing the following definition:

\begin{definition}
A compact space $\mathcal{F}^* \subseteq \mathcal{F}_{T^*}^{\mathrm{all}}$ of GENEOs $(F^*, T^*) : (\Phi^*, G^*) \to (\Psi^*, H^*)$ with $\mathrm{dom}(\Phi^*) = X^*$ and $\mathrm{dom}(\Psi^*) = Y^*$ is said to be a \textit{compactification} of a space $\mathcal{F} \subseteq \mathcal{F}_{T}^{\mathrm{all}}$ of GENEOs $(F, T) : (\Phi, G) \to (\Psi, H)$ with $\mathrm{dom}(\Phi) = X$ and $\mathrm{dom}(\Psi) = Y$, if the perception pairs $(\Phi^*, G^*)$ and $(\Phi^*, G^*)$ are compactifications of $(\Phi, G)$ and $(\Psi, H)$ respectively, and there is an isometric embedding $f$ of $\mathcal{F}$ in $\mathcal{F}^{*}$ as a dense subspace, such that the following commutativity conditions are satisfied: $i_2 \circ F = f(F) \circ i_1$, for each $F \in \mathcal{F}$, and $k_2 \circ T = T^* \circ k_1$.
\end{definition}


\begin{theorem}
Every collectionwise surjective space $\mathcal{F} \subseteq \mathcal{F}_{T}^{\mathrm{all}}$ of GENEOs $(F, T) : (\Phi, G) \to (\Psi, H)$ with $\mathrm{dom}(\Phi) = X$ and $\mathrm{dom}(\Psi) = Y$ admits a compactification $\mathcal{F}^*$, provided the data sets $\Phi$ and $\Psi$ are totally bounded and endow $X$ and $Y$ with metric structures.
\end{theorem}
\begin{proof}
It follows from Theorem \ref{CompPercPa} and Propositions \ref{ComGNT}, \ref{TdHtFComp}, \ref{ComGN4}, and \ref{ComGNF}, by setting
$F^* := \overline{\hat{F}}$, $T^* := \overline{\hat{T}}$, and $\mathcal{F}^* := \mathrm{cl}(\mathcal{F}_3) \subseteq \mathcal{F}_{\overline{\hat{T}}}^{\mathrm{all}}$.
\end{proof}


\section{Discussion}\label{Dis}
In this paper, we have shown that when the spaces of measurements are totally bounded and large enough to ensure that any two points can be distinguished by our measurements, we can always assume that we are considering compact perception pairs and compact spaces of GENEOs, provided that the set of our operators is collectionwise surjective. This result makes available a sound basis for further research concerning spaces of GENEOs, and paves the way for possible applications of the theory.

Of course, the computation costs might be higher while working with compactifications, but that should not be considered to be a drawback. In fact, in practical applications, one does not necessarily need to work with compactifications in an explicitly concrete manner. The mere recognition that certain spaces of GENEOs can be nicely embedded in compact ones is all that one needs most of the time.

Our research has raised several questions as well. For example, it is not clear whether the assumption of collectionwise surjectivity could be removed or made milder. Furthermore, we could wonder if our approach could be extended to the case when $X$ and $Y$ are endowed with a pseudo-metric instead of a metric structure, thereby extending the range of applicability of our constructions. We are planning to follow these lines of research in the near future.



\backmatter

\bmhead{Acknowledgments}

The author is greatly indebted to Patrizio Frosini for supervising the project. He thanks Massimo Ferri for his precious advice and support. Finally, he thanks Nicola Quercioli for many helpful suggestions.


\section*{Declarations}

This research has been partially supported by INdAM-GNSAGA.  


\begin{appendices}

\section{Supplementary Proofs}\label{App}

For the sake of completeness, we recall here the proofs of some results reported in Section \ref{BR} that have been given only in \cite{NQTh}.

\begin{proof}[Proof of Proposition \ref{SigNonExp}]
Let $\varphi \in \Phi$ and $x_1, x_2 \in X$. Then \[\vert \varphi(x_1)-\varphi(x_2) \vert \le \sup_{\varphi' \in \Phi} \vert \varphi'(x_1)-\varphi'(x_2) \vert = D_X(x_1,x_2). \] 
So $\varphi : X \to \mathbb{R}$ is non-expansive.
\end{proof}

\begin{proof}[Proof of Theorem \ref{PhiBddXBdd}]
	It suffices to show that every sequence $(x_i)_{i\in \mathbb{N}}$ in $X$ admits a Cauchy subsequence \cite{Gaal}.
	Let us consider an arbitrary sequence $(x_i)_{i \in \mathbb{N}}$ in $X$ and an arbitrarily small $\varepsilon > 0$. Since $\Phi$ is totally bounded, we can find a finite subset $\Phi_{\varepsilon} = \{ \varphi_1, \dots , \varphi_n \}$ such that $\Phi = \bigcup_{i=1}^n B_{\Phi}(\varphi_i, \varepsilon)$, where
	$B_{\Phi}(\varphi,\varepsilon)=\{\varphi' \in \Phi : D_\Phi(\varphi',\varphi) < \varepsilon \}$.
	In particular, we can say that for any $\varphi \in \Phi$ there exists $\varphi_{\bar k} \in \Phi_\varepsilon$ such that $\|\varphi - \varphi_{\bar k}\|_{\infty} < \varepsilon$. Now, we consider the real sequence $(\varphi_1(x_i))_{i\in \mathbb{N}}$ that is bounded because all the functions in $\Phi$ are bounded. From Bolzano-Weierstrass Theorem
	it follows that we can extract a convergent subsequence $(\varphi_1(x_{i_h}))_{h \in \mathbb{N}}$. Then we consider the sequence $(\varphi_2(x_{i_h}))_{h \in \mathbb{N}}$. Since $\varphi_2$ is bounded, we can extract a convergent subsequence $(\varphi_2(x_{i_{h_t}}))_{t \in \mathbb{N}}$. We can repeat the same argument for any $\varphi_k \in \Phi_{\varepsilon}$. Thus, we obtain a subsequence $(x_{p_{j}})_{j \in \mathbb{N}}$ of $(x_i)_{i \in \mathbb{N}}$, such that $(\varphi_k(x_{p_j}))_{j \in \mathbb{N}}$ is a real convergent sequence for any $k \in \{1, \dots, n \}$, and hence a Cauchy sequence in $\mathbb{R}$. Moreover, since $\Phi_{\varepsilon}$ is a finite set, there exists an index $\bar{\jmath}$ such that for any $k \in \{1, \dots, n \}$ we have that

	\[ \left \vert \varphi_k(x_{p_r}) - \varphi_k(x_{p_s}) \right \vert < \varepsilon, \ \ \mathrm{for \ all} \ r,s \geq \bar{\jmath}. \]
	
	We observe that $\bar{\jmath}$ does not depend on $k$, but only on $\varepsilon$ and $\Phi_{\varepsilon}$.
	
	In order to prove that $(x_{p_{j}})_{j \in \mathbb{N}}$ is a Cauchy sequence in $X$, we observe that for any $r,s \in \mathbb{N}$ and any $\varphi \in \Phi$, by choosing a $k$ such that $\|\varphi-\varphi_k\|_\infty < \varepsilon$ we have:
	\begin{align*}
	\left \vert \varphi(x_{p_r}) - \varphi(x_{p_s}) \right \vert &= \left \vert\varphi(x_{p_r}) - \varphi_k(x_{p_r}) + \varphi_k(x_{p_r}) - \varphi_k(x_{p_s}) + \varphi_k(x_{p_s}) - \varphi(x_{p_s}) \right \vert \nonumber\\
	& \leq  \left \vert \varphi(x_{p_r}) - \varphi_k(x_{p_r}) \right \vert + \left \vert \varphi_k(x_{p_r}) - \varphi_k(x_{p_s}) \right \vert + \left \vert \varphi_k(x_{p_s}) - \varphi(x_{p_s}) \right \vert \nonumber\\
	& \leq  \|\varphi - \varphi_k\|_{\infty} + \left \vert \varphi_k(x_{p_r}) - \varphi_k(x_{p_s}) \right \vert + \|\varphi_k - \varphi\|_{\infty}. 
	\end{align*}
	
	It follows that $\left \vert \varphi(x_{p_r}) - \varphi(x_{p_s})\right \vert < 3\varepsilon$ for every $\varphi \in \Phi$ and every $r,s \geq \bar{\jmath}$.
	Thus, $\sup_{\varphi \in \Phi} \left \vert \varphi(x_{p_r}) - \varphi(x_{p_s}) \right \vert = D_X(x_{p_r},x_{p_s}) \le 3\varepsilon$. Hence, the subsequence $(x_{p_j})_{j \in \mathbb{N}}$ is a Cauchy sequence in $X$, and the theorem is proved.
\end{proof}

\begin{proof}[Proof of Theorem \ref{PhiBddGBdd}]
	Let $(g_i)_{i\in \mathbb{N}}$ be a sequence in $G$ and take a real number $\varepsilon > 0$. Given that $\Phi$ is totally bounded, we can find a finite subset $\Phi_{\varepsilon} = \{ \varphi_1, \dots , \varphi_n \}$ such that for every $\varphi\in\Phi$ there exists $\varphi_h \in \Phi_\varepsilon$ for which $D_\Phi(\varphi_h,\varphi) < \varepsilon$. 
	
	Let us consider the sequence $(\varphi_1g_i)_{i\in \mathbb{N}}$ in $\Phi$. Since $\Phi$ is totally bounded, we can extract a Cauchy subsequence $(\varphi_1g_{i_h})_{h \in \mathbb{N}}$ \cite{Gaal}. Then we consider the sequence $(\varphi_2g_{i_h})_{h \in \mathbb{N}}$. Again, we can extract a Cauchy subsequence $(\varphi_2g_{i_{h_t}})_{t \in \mathbb{N}}$. We can repeat the same argument for any $\varphi_k \in \Phi_{\varepsilon}$. Thus, we are able to extract a subsequence $(g_{i_j})_{j\in \mathbb{N}}$ of $(g_i)_{i\in \mathbb{N}}$ such that $(\varphi_k  g_{i_j})_{j\in \mathbb{N}}$ is a Cauchy sequence for any $k \in \{1, \dots , n\}$. For the finiteness of set $\Phi_\varepsilon$, we can find an index $\bar{\jmath}$ such that for any $k \in \{1, \dots , n\}$
	
	\[	D_\Phi(\varphi_k  g_{i_r},\varphi_k g_{i_s}) < \varepsilon, \ \mathrm{for \ every} \ s,r \geq \bar{\jmath}.\]

	In order to prove that $(g_{i_{j}})_{j\in \mathbb{N}}$ is a Cauchy sequence, we observe that for any $\varphi\in\Phi$, any $\varphi_k\in\Phi_\varepsilon$, and any $r,s \in \mathbb{N}$ we have
	\begin{align*}
		D_\Phi(\varphi g_{i_r},\varphi  g_{i_s}) 
		& \leq D_\Phi(\varphi  g_{i_r},\varphi_k g_{i_r}) + D_\Phi(\varphi_k  g_{i_r},\varphi_k  g_{i_s}) + D_\Phi(\varphi_k g_{i_s},\varphi g_{i_s})\nonumber \\
		& =  D_\Phi(\varphi,\varphi_k) + D_\Phi(\varphi_k g_{i_r},\varphi_k  g_{i_s}) + D_\Phi(\varphi_k, \varphi).
	\end{align*}
	
	We observe that $\bar{\jmath}$ does not depend on $\varphi$, but only on $\varepsilon$ and $\Phi_{\varepsilon}$. By choosing a $\varphi_k\in\Phi_\varepsilon$ such that $D_\Phi(\varphi_k,\varphi) < \varepsilon$, we get $D_\Phi(\varphi g_{i_r},\varphi  g_{i_s}) < 3 \varepsilon$ for every $\varphi \in \Phi$ and every $r,s \geq \bar{\jmath}$.
	Thus, $D_{\mathrm{Aut}}(g_{i_r},g_{i_s}) \le 3\varepsilon$. Hence, the sequence $(g_{i_j})_{j\in \mathbb{N}}$ is a Cauchy sequence. Therefore, $G$ is totally bounded.
\end{proof}

\begin{proof}[Proof of Proposition \ref{IsoXComp}]
Let $C(X,X)$ denote the metric space of all continuous self-maps of $X$ with respect to the metric $d_{\infty}$ given by \[d_\infty(f,g):= \sup_{x \in X}D_X\left(f(x),g(x)\right), \ \mathrm{for\ every\ } f, g \in C(X,X). \] It suffices to show that $\mathrm{Iso}(X)$ is closed in $C(X,X)$, as it is relatively compact by Arzelà-Ascoli theorem \cite{ArAsc}. 
Let $(f_i)_{i \in \mathbb{N}}$ be a sequence in $\mathrm{Iso}(X)$ that converges to some $f \in C(X,X)$; we show that $f \in \mathrm{Iso}(X)$. Note that $f(x)= \lim_{i \to \infty} f_i(x)$ with respect to $D_X$, for each $x \in X$; indeed, \[0\le\lim_{i \to \infty}D_X( f(x),f_i(x)) \le \lim_{i \to \infty} d_\infty( f,f_i)=0.\] So, 
\begin{align*}
	D_X(f(x),f(y)) & = D_X(\lim_{i \to \infty}f_i(x),\lim_{i \to \infty}f_i(y)) \\ 
	& = \lim_{i \to \infty}D_X(f_i(x),f_i(y)) \\
	& = \lim_{i \to \infty}D_X(x,y) \\
	& =D_X(x,y),
\end{align*}
whence $f$ preserves $D_X$.

It immediately follows that $f$ is injective. As for surjectivity, let $x_{0}$ be an arbitrary point of $X$; we show that $x_0 \in f(X)$. Consider the sequence $(x_n)_{n \in \mathbb{N}}$ defined by setting $x_{n+1} := f(x_n)$. Since $X$ is compact, $(x_n)_{n \in \mathbb{N}}$ admits a converging subsequence $(x_{n_i})_{i \in \mathbb{N}}$. Let $\varepsilon > 0$ be an arbitrary real number. Then there is an $n_0 \in \mathbb{N}$ such that $D_X(x_{n_i}, x_{n_j}) < \varepsilon$ for every $i, j \geq n_0$. If $n_j \geq n_i$, then $D_X(x_{n_i}, x_{n_j}) = D_X(x_0, x_{n_j-n_i})$,  as $f$ preserves $D_X$. Hence, $D_X(x_0,f(X)):=\inf_{x \in f(X)}D_X(x_0,x)\le D_X(x_0,x_{n_j-n_i})= D_X(x_{n_i},x_{n_j}) < \varepsilon.$ From the arbitrariness of $\varepsilon$, it follows that $D_X(x_0,f(X))=0$. As $f$ preserves $D_X$, it is continuous, and $f(X)$ then is compact. In particular, $f(X)$ is closed, and hence $x_0 \in f(X)$.
\end{proof}

\begin{proof}[Proof of Theorem \ref{AutComp}]
For the sake of conciseness, we will rephrase the proof given in \cite{NQTh}. Consider the collection $\mathcal{H}$ of all non-empty compact subsets of the space $\mathbf{NE}(X, \mathbb{R})$ of all real-valued non-expansive functions on $(X,D_X)$, endowed with the distance induced by the uniform norm.  Of course, by Proposition \ref{SigNonExp}, $\mathbf{NE}(X, \mathbb{R})\supseteq \Phi$.
We know that $(\mathcal{H}, d_\mathcal{H})$ is a metric space, where $d_\mathcal{H}$ is the usual Hausdorff distance \cite{Eng}.
If $g \in \mathrm{Iso}(X)$, then the map $R_g:\Phi\to\mathbb{R}_b^X$ that takes $\varphi$ to $\varphi g$ is continuous (it  indeed preserves the max-morm distance), and hence $\Phi g:=\{\varphi g, \varphi \in \Phi\} \in \mathcal{H}$. 

We now observe that if
$g,h \in \mathrm{Iso}(X)$, then
    \begin{align*}
		d_\mathcal{H}(\Phi g, \Phi h)& := \max \left\{\sup_{\varphi \in \Phi g}\inf_{\psi \in \Phi h}\| \varphi - \psi\|_\infty , \sup_{\psi \in \Phi h}\inf_{\varphi \in \Phi g}\| \varphi - \psi\|_\infty\right\}\\
		& = \max \left\{\sup_{\varphi \in \Phi }\inf_{\psi \in \Phi }\| \varphi g - \psi h\|_\infty , \sup_{\psi \in \Phi }\inf_{\varphi \in \Phi }\| \varphi g - \psi h \|_\infty\right\}\\
		& \le \max \left\{\sup_{\varphi \in \Phi }\| \varphi g - \varphi h\|_\infty , \sup_{\psi \in \Phi }\| \psi g - \psi h \|_\infty\right\}\\
		& = \sup_{\varphi \in \Phi }\| \varphi g - \varphi h\|_\infty\\
		& = \sup_{\varphi \in \Phi }\sup_{x \in X} \vert \varphi g(x) - \varphi h(x) \vert\\
		& = \sup_{x \in X} \sup_{\varphi \in \Phi }\vert \varphi g(x) - \varphi h(x) \vert\\
		& = \sup_{x \in X} D_X(g(x),h(x)) \\
		& = d_\infty(g,h).
    \end{align*}

Therefore, the map $\chi:\mathrm{Iso}(X)\to \mathcal{H}$ that takes $g$ to $\Phi g$ is non-expansive and hence continuous.
Since $\mathrm{Aut}_{\Phi}(X)=\chi^{-1}(\Phi)$, such a group is the preimage of a closed set under a continuous function. It follows that it is closed in $\mathrm{Iso}(X)$, and hence compact. 
\end{proof}





\end{appendices}






\bigskip


\end{document}